\documentclass{article}
\usepackage{graphicx}
\usepackage{amsthm}
\usepackage{amsfonts, amsmath, amssymb, amsthm, amsrefs}
\usepackage[hidelinks]{hyperref}
\usepackage{cleveref}
\usepackage[utf8x]{inputenc}
\usepackage{url}

\usepackage{tikz}
\usepackage{tkz-graph}
\usetikzlibrary{arrows.meta}

\usepackage{authblk}

\newcommand{\Psf}{\mathsf{P}}
\newcommand{\Qsf}{\mathsf{Q}}

\def\F{\mathcal{F}}

\def\C{\mathcal{C}}

\def\Q{\mathcal{Q}}
\def\W{\mathcal{W}}
\def\P{\mathcal{P}}
\def\R{\mathcal{R}}
\def\Sc{\mathcal{S}}
\def\M{\mathcal{M}}

\def\Q{\mathcal{Q}}

\newcommand{\cs}{2^{\NN}}

\newcommand{\uh}{\upharpoonright}

\newcommand{\converge}{\!\!\downarrow}
\newcommand{\diverge}{\!\!\uparrow}

\newcommand{\qvdash}{\operatorname{{?}{\vdash}}}
\newcommand{\nqvdash}{\operatorname{{?}{\nvdash}}}

\newcommand{\RCA}[0]{\mathsf{RCA}}
\newcommand{\WKL}[0]{\mathsf{WKL}}
\newcommand{\ACA}[0]{\mathsf{ACA}}

\newcommand{\EM}[0]{\mathsf{EM}}

\newcommand{\RT}[0]{\mathsf{RT}}

\newcommand{\COH}[0]{\mathsf{COH}}
\newcommand{\BRT}[0]{\mathsf{BRT}}
\newcommand{\WWKL}[0]{\mathsf{WWKL}}

\newcommand{\NN}[0]{\mathbb{N}}

\newcommand{\card}{\operatorname{card}}

\newcommand{\rank}{\mathtt{rk}}

\def\qt#1{``#1''}%

\newcommand{\ol}[1]{\overline{#1}}

\title{Bounded Ramsey's theorem for triples in computability theory}
\date{\today}

\newtheorem*{statement}{Statement}
\newtheorem{theorem}{Theorem}
\numberwithin{theorem}{section}

\newtheorem{lemma}[theorem]{Lemma}

\newtheorem{proposition}[theorem]{Proposition}
\newtheorem{remark}[theorem]{Remark}
\newtheorem{definition}[theorem]{Definition}
\newtheorem{corollary}[theorem]{Corollary}
\newtheorem{example}[theorem]{Example}

\makeatletter
\newtheorem*{rep@theorem}{\rep@title}
\newcommand{\newreptheorem}[2]{%
\newenvironment{rep#1}[1]{%
 \def\rep@title{#2 \ref{##1}}%
 \begin{rep@theorem}}%
 {\end{rep@theorem}}}
\makeatother

\newreptheorem{theorem}{Theorem}

\AtBeginDocument{%
   \def\MR#1{}
}

\usepackage{xcolor}

\author[1]{Ludovic Patey}
\affil[1]{CNRS, Institut de Math\'{e}matiques de Jussieu--Paris Rive Gauche, Universit\'{e} Paris Cit\'{e} -- B\^{a}timent Sophie Germain,
8 Place Aur\'{e}lie Nemours, 75205 Paris Cedex 13, France\\
\href{mailto:ludovic.patey@computability.fr}{ludovic.patey@computability.fr}\\
\url{https://ludovicpatey.com}
\medskip
}

\author[2]{Paul Shafer}
\affil[2]{School of Mathematics, University of Leeds, Leeds, LS2 9JT, United Kingdom\\
\href{mailto:p.e.shafer@leeds.ac.uk}{p.e.shafer@leeds.ac.uk}\\
\url{https://peshafer.github.io}}

\begin{document}

\maketitle

\begin{abstract}
We study a restriction of Ramsey's theorem for 2-coloring of triples, in which homogeneous sets for color~1 are of bounded size ($\BRT^3_2$). We prove that the computational content of this statement is very close to Ramsey's theorem for pairs ($\RT^2_2)$, in that it satisfies the same known computability-theoretic upper bounds, but that $\BRT^3_2$ is not computably-reducible to $\RT^2_2$, even when allowing multiple applications of $\RT^2_2$.
\end{abstract}

\section{Introduction}

In this article, we study a bounded version of Ramsey's theorem for triples from the viewpoint of reverse mathematics and computable reductions. Ramsey's theorem is a fundamental theorem in combinatorics at the origins of Ramsey theory. Given a set $X \subseteq \NN$ and $n \in \NN$, we write $[X]^n$ for the collection of all unordered $n$-tuples over~$X$. For every $n, k \geq 1$, Ramsey's theorem for $n$-tuples and $k$-colors ($\RT^n_k$) states that every coloring $f : [\NN]^n \to \{0, \dots, k-1\}$ admits an infinite set $H \subseteq \NN$ such that $f$ uses exactly one color~$i$ on $[X]^n$. We then say that $H$ is \emph{$f$-homogeneous (for color~$i$)}. 

In the same way that most applications of Borel determinacy require only low levels in the Borel hierarchy, many applications of Ramsey's theorem in other areas of mathematics involve degenerate colorings. For instance, Ramsey's theorem for transitive colorings is used in a factorization theorem in automata theory~\cite{murakami2014ramseyan}. In this article, we study another restriction of Ramsey's theorem for 2-colorings, in which the homogeneous sets of color~1 are of bounded size. 

\begin{statement}[Bounded Ramsey's theorem]
$\BRT^n_{2, \ell}$ is the statement \qt{Every coloring $f : [\NN]^n \to 2$ which has no $f$-homogeneous set for color~1 of size~$\ell$ has an infinite $f$-homogeneous set for color~0.} $\BRT^n_2$ is the statement $\forall \ell \BRT^n_{2,\ell}$.
\end{statement}

Such bounded colorings appear naturally in the proof that $\RT^{n+1}_2$ implies $\forall k\RT^n_k$ over~$\RCA_0$, for example. The systematic study of this restriction was started by Frittaion, and it was pursued by Sold\`a~\cite[Section 4.1]{solda2021calibrating} and, more recently, by Le Houérou and Patey~\cite{houerou2025reverse}, who studied bounded colorings of pairs.

Ramsey's theorem is extensively studied in reverse mathematics~\cites{jockusch_ramseys_1972,simpson_subsystems_2009,seetapun1995strength}. For $n \geq 3$, $\RT^n_2$ is equivalent to the Arithmetic Comprehension Axiom ($\ACA_0$) over~$\RCA_0$, but this collapsing does not reflect the  computability-theoretic content of Ramsey's theorem, which is better explained in terms of pure computability theory.

The computational strength of Ramsey's theorem for $(n+1)$-tuples ($\RT^{n+1}_2$) can be roughly understood as the \qt{jump} of the strength of its version for $n$-tuples ($\RT^n_2$). Indeed, by Shoenfield's limit lemma, every $\emptyset'$-computable coloring of $n$-tuples can be represented as a computable coloring of $(n+1)$-tuples, and conversely, every stable computable coloring of $(n+1)$-tuples induces a $\emptyset'$-computable coloring of $n$-tuples by considering the limit coloring. Based on this correspondence, the computability-theoretic bounds of Ramsey's theorem for pairs translate to similar bounds for $n$-tuples by increasing the level in the arithmetic hierarchy accordingly. For instance, Jockusch~\cite{jockusch_ramseys_1972} proved that for every~$n \geq 2$, there exists a computable instance of $\RT^n_2$ with no $\Sigma^0_n$ solution, while every computable instance of $\RT^n_2$ admits a $\Pi^0_n$ solution. Moreover, again by Jockusch~\cite{jockusch_ramseys_1972}, for every~$n \geq 2$, there exists a computable instance of $\RT^n_2$ such that every solution computes $\emptyset^{(n-2)}$; while Seetapun~\cite{seetapun1995strength} and Cholak, Jockusch and Slaman~\cite{cholak2001strength} proved that for every~$n \geq 2$, every computable instance of $\RT^n_2$ has a solution which does not compute $\emptyset^{(n-1)}$.

There exists a similar hierarchy for the bounded version of Ramsey's theorem, with a subtlety on the bound: We shall see that $\BRT^{n+1}_{2,\ell+1}$ can be thought of as the jump version of $\BRT^n_{2,\ell}$. However, the $\BRT_2$-hierarchy seems to be shifted compared to the $\RT_2$-hierarchy, in that the computational strength of $\BRT^{n+1}_2$ is very close to that of $\RT^n_2$. For instance, $\BRT^2_2$, like $\RT^1_2$, is a computably true statement, which means that every computable instance admits a computable solution. It follows that the strength of $\BRT^2_2$ is better understood in terms of the amount of induction needed to prove the existence of a solution~\cite{houerou2025reverse}. Bounded Ramsey's theorem for triples is the first level of this hierarchy which is not computably true, and therefore is relevant for a computability-theoretic analysis. This is our main object of study in this article.

\subsection{Reverse mathematics and computable reductions}

Reverse mathematics is a foundational program started by Harvey Friedman~\cite{friedman1975systems}, whose goal is to find optimal axioms to prove ordinary theorems. It uses the framework of subsystems of second-order arithmetic, with a base theory $\RCA_0$ capturing \qt{computable mathematics.} More precisely, $\RCA_0$ is composed of Robinson's arithmetic together with the $\Sigma^0_1$-induction scheme and the $\Delta^0_1$-comprehension scheme. See Simpson~\cite{simpson_subsystems_2009} or Dzhafarov and Mummert~\cite{dzhafarov_reverse_2022} for an introduction to reverse mathematics.

Structures in the language of second-order arithmetic are of the form $\M = (M, S, 0, 1, <, +, \cdot)$, where $M$ denotes the set of integers of the structure, and $S \subseteq \P(M)$ is the collection of its reals, also known as its second-order part. An \emph{$\omega$-structure} is a structure for which $M$ consists of the standard integers, together with the usual interpretations of the constants, relations, and operations. Such structures are of particular interest, especially when proving a separation, as they are closer to the intended model. An $\omega$-structure is fully characterized by its second-order part $S$, and $\omega$-models of $\RCA_0$ admit a purely computability-theoretic characterization. Indeed, by Friedman, an $\omega$-structure is a model of~$\RCA_0$ if and only if its second-order part~$S$ is a \emph{Turing ideal}, that is, a non-empty collection of sets which is downward-closed under Turing reduction $(\forall X \in S, \forall Y \leq_T X, Y \in S)$ and closed under Turing join ($\forall X, Y \in S, X \oplus Y = \{ 2n : n \in X \} \cup \{ 2n+1 : n \in Y \} \in S$).

Many statements~$\Psf$ studied in reverse mathematics are of the form 
$$\forall X(\Phi(X) \to \exists Y \Psi(X, Y))$$
where $\Phi$ and $\Psi$ are arithmetic formulas. Such statements, called \emph{$\Pi^1_2$-problems}, can be seen as mathematical problems, whose \emph{instances} are sets~$X$ such that $\Phi(X)$ holds, and a set $Y$ is a \emph{solution} to an instance~$X$ if $\Psi(X, Y)$ holds. For example, an instance of $\BRT^n_{2,\ell}$ is a coloring $f : [\NN]^n \to 2$ with no $f$-homogeneous set for color~1 of size~$\ell$, and a solution to~$f$ is an infinite $f$-homogeneous set for color~0. An $\omega$-model~$\M$ of~$\RCA_0$ satisfies a $\Pi^1_2$-problem $\Psf$ if every $\Psf$-instance in $\M$ admits a $\Psf$-solution in~$\M$. We say that a $\Pi^1_2$-problem~$\Psf$ is \emph{$\omega$-reducible} to another $\Pi^1_2$-problem $\Qsf$ (written $\Psf \leq_\omega \Qsf$ if every $\omega$-model of~$\RCA_0 + \Qsf$ is an $\omega$-model of~$\Psf$. In particular, if $\RCA_0 \vdash \Qsf \to \Psf$, then $\Psf \leq_\omega \Qsf$.

The $\omega$-reduction is a bridge between the proof-theoretic notion of implication over $\RCA_0$ and the computability-theoretic realm. Hirschfeldt and Jockusch~\cite{hirschfeldt_notions_2016} gave a characterization of $\omega$-reducibility in terms of games, enabling finer control on the number of applications of $\Psf$ in the reduction $\Qsf \leq_\omega \Psf$.

\begin{definition}[Hirschfeldt and Jockusch~\cite{hirschfeldt_notions_2016}]
For problems $\Psf$ and $\Qsf$, the \emph{reduction game} $G(\Qsf \to \Psf)$ is a two-player game that proceed as follows.
If at any point, one of the players does not have a legal move, then the game ends with a victory for the other player.
On the first move, Player~1 plays a $\Psf$-instance~$X_0$, and Player~2 either plays an $X_0$-computable $\Psf$-solution to~$X_0$ and declares victory, in which case the game ends, or responds with an $X_0$-computable $\Qsf$-instance~$Y_1$.
For $n > 1$, on the $n$\textsuperscript{th} move, Player~1 plays a $\Qsf$-solution~$X_{n-1}$ to the $\Qsf$-instance~$Y_{n-1}$. Then Player~2 either plays a $(\bigoplus_{i < n} X_i)$-computable $\Psf$-solution to~$X_0$ and declares victory, in which case again the game ends, or plays a $(\bigoplus_{i < n} X_i)$-computable $\Qsf$-instance~$Y_n$.
Player~2 wins this play of the game if it ever declares victory or if Player~1 has no legal move at some point of the game. Otherwise, Player~1 wins.
\end{definition}

Hirschfeldt and Jockusch~\cite{hirschfeldt_notions_2016} proved that $\Psf \leq_\omega \Qsf$ if and only if Player~2 has a winning strategy for the game $G(\Qsf \to \Psf)$. Given~$k \in \NN$, let $\Psf \leq_\omega^k \Qsf$ hold if Player~2 has a winning strategy for the game $G(\Qsf \to \Psf)$ that guarantees victory in $k+1$ or fewer moves. By the correspondence between $\omega$-reducibility and provability over $\RCA_0$, showing that $\Psf \not\leq_\omega^k \Qsf$ could be understood as stating that there is no proof of $\Psf$ over $\RCA_0 + \Qsf$ with at most $k$ applications of~$\Qsf$. However, this intuition should be tempered by a result from Hirst and Mummert~\cite{hirst2019using}, who showed that $\RCA_0 \vdash \RT^n_2 \to \RT^n_k$ for any standard $n, k \in \NN$ with only one application of $\RT^n_2$, while it is known by Patey~\cite{patey2016reverse} that $\RT^n_2 \not \leq_c \RT^n_3$.

\subsection{Stability and cohesiveness}\label[section]{sec:stability-cohesiveness}

In order to better understand the computational strength of Ramsey's theorem for pairs, Cholak, Jockusch and Slaman~\cite{cholak2001strength} decomposed $\RT^2_2$ into its stable and cohesive versions.
A coloring $f : [\NN]^{n+1} \to k$ is \emph{stable} if for every $\vec{x} \in [\NN]^n$, $\lim_y f(\vec{x}, y)$ exists. The \emph{limit coloring} of~$f$ is the coloring $g : [\NN]^n \to k$ defined by $g(\vec{x}) = \lim_y f(\vec{x}, y)$. 

As mentioned, Shoenfield's limit lemma forms a bridge between stable Ramsey's theorem for computable colorings of $(n+1)$-tuples and Ramsey's theorem for $\Delta^0_2$ colorings of $n$-tuples. We now give a formal correspondence between the bounded counterparts of these statements. 

\begin{lemma}
Let $f : [\NN]^{n+1} \to 2$ be a computable stable instance of $\BRT^{n+1}_{2, \ell+1}$.
Its limit coloring is a $\Delta^0_2$ instance of $\BRT^n_{2, \ell}$.
\end{lemma}
\begin{proof}
Let $g : [\NN]^n \to 2$ be the limit coloring of~$f$. By Shoenfield's limit lemma, $g$ is $\Delta^0_2$. Suppose for a contradiction that there is a $g$-homogeneous set~$H$ for color~1 of size~$\ell$. Then for every~$\vec{x} \in H$, $\lim_y f(\vec{x}, y) = 1$, so for some sufficiently large~$y$, $H \cup \{y\}$ is $f$-homogeneous for color~1 and has size~$\ell+1$, which is a contradiction.
\end{proof}

\begin{lemma}\label[lemma]{lem:delta2-stable-brt}
For every $\Delta^0_2$ instance $g : [\NN]^n \to 2$ of $\BRT^n_{2, \ell}$,
there is a stable computable instance of $\BRT^{n+1}_{2, \ell+1}$ whose limit coloring is~$g$.
\end{lemma}
\begin{proof}
By Shoenfield's limit lemma, let $h : [\NN]^{n+1} \to 2$ be a computable stable coloring whose limit coloring is~$g$.  Fix a computable ordering of $[\NN]^n$ of type $\omega$.  Compute a coloring $f : [\NN]^{n+1} \to 2$ as follows.  At step $s$, consider the elements $\vec{x}$ of $[s]^n$ in order.  For each $\vec{x} \in [s]^n$, set $f(\vec{x}, s) = 1$ if $h(\vec{x}, s) = 1$ and there is not an $F \subseteq s$ of size $\ell$ where the maximum element of $[F]^n$ is $\vec{x}$ and we have already set $f(\vec{z}, s) = 1$ for all $\vec{z} \in [F]^n$ with $\vec{z} < \vec{x}$.  Otherwise set $f(\vec{x}, s) = 0$.  Consider a set $H$ of size $\ell+1$, and let $s$ be its maximum element.  Let $\vec{x}$ be the maximum element of $[H \setminus \{s\}]^n$.  If $h(\vec{x}, s) = 0$, then $f(\vec{x}, s) = 0$.  If $h(\vec{x}, s) = 1$, then either there is a $\vec{z} < \vec{x}$ in $[H \setminus \{s\}]^n$ with $f(\vec{z}, s) = 0$, or we set $f(\vec{x}, s) = 0$.  In all cases, we see that $H$ is not $f$-homogeneous for color $1$.

For every $\vec{x} \in [\NN]^n$, $\lim_s h(\vec{x}, s) = 0$ implies that $\lim_s f(\vec{x}, s) = 0$.  It remains to show that for every $\vec{x} \in [\NN]^n$, $\lim_s h(\vec{x}, s) = 1$ implies that $\lim_s f(\vec{x}, s) = 1$.  Suppose for a contradiction that $\vec{x} \in [\NN]^n$ is least such that $\lim_s h(\vec{x}, s) = 1$ but that there are infinitely many $s$ with $f(\vec{x}, s) = 0$.  Thus for infinitely many $s$, there is a set $F$ of size $\ell$ where the maximum element of $[F]^n$ is $\vec{x}$, and $f(\vec{z}, s) = 1$ for all $\vec{z} \in [F]^n$ with $\vec{z} < \vec{x}$.  There are only finitely many sets $F$ of size $\ell$ where the maximum element of $[F]^n$ is $\vec{x}$, so there is a fixed such $F$ where, for infinitely many $s$, $f(\vec{z}, s) = 1$ for all $\vec{z} \in [F]^n$ with $\vec{z} < \vec{x}$.  By the minimality of $\vec{x}$, it must be that $1 = \lim_s f(\vec{z}, s) = \lim_s h(\vec{z}, s) = g(\vec{z})$ for all $\vec{z} \in [F]^n$ with $\vec{z} < \vec{x}$.  By assumption, $1 = \lim_s h(\vec{x}, s) = g(\vec{x})$ as well.  Thus $F$ is a contradictory $g$-homogeneous set for color $1$ of size $\ell$.  Thus $f$ is a stable computable instance of $\BRT^{n+1}_{2, \ell+1}$ whose limit coloring agrees with that of $h$ and therefore is $g$.
\end{proof}

Given an infinite sequence of sets $\vec{R} = R_0, R_1, \dots$, an infinite set $C \subseteq \NN$ is \emph{$\vec{R}$-cohesive} if for every $n$, $C \subseteq^* R_n$ or $C \subseteq^* \overline{R}_n$, where $\subseteq^*$ means \qt{included up to finitely many elements.} Cohesiveness is an essential tool in restricting a computable coloring to a stable sub-domain, as follows:

\begin{lemma}[Cholak, Jockusch and Slaman~\cite{cholak2001strength}]\label[lemma]{lem:coh-makes-stable}
Let $f : [\NN]^{n+1} \to k$ be a computable coloring.
There exists a uniformly computable sequence of sets $\vec{R}$ 
such that for every $\vec{R}$-cohesive set~$C$, $f$ is stable on $[C]^{n+1}$.
That is, for every~$\vec{x} \in [C]^n$, $\lim_{y \in C} f(\vec{x}, y)$ exists.
\end{lemma}
\begin{proof}
For every~$\vec{x} \in [\NN]^n$ and $i < k$, let $R_{\vec{x},i} = \{ y : f(\vec{x}, y) = i \}$. Let $C$ be an infinite $\vec{R}$-cohesive set, where $\vec{R} = \{ R_{\vec{x},i} \}_{\vec{x} \in [\NN]^n, i < k}$. Let $\vec{x} \in [C]^n$. Suppose that $\lim_{y \in C} f(\vec{x}, y)$ does not exist. Then there are some~$i \neq j$ such that $C \cap R_{\vec{x}, i}$ and $C \cap R_{\vec{x}, j}$ are both infinite. By the $\vec{R}$-cohesiveness of~$C$, $C \subseteq^* R_{\vec{x}, i} \cap R_{\vec{x}, j} = \emptyset$, which is a contradiction.
\end{proof}

The strength of computing an infinite $\vec{R}$-cohesive set for a computable sequence of sets $\vec{R}$ is well-understood in terms of jump computation~\cites{jockusch_cohesive_1993,belanger_conservation_2015}.
Based on the notions of stability and cohesiveness, our study of computable instances of $\BRT^3_2$ will essentially consist in proving theorems about $\Delta^0_2$ instances of $\BRT^2_2$, in the same way that the study of Ramsey's theorem for pairs mainly relies on the study of $\Delta^0_2$ instances of the pigeonhole principle.

\subsection{Main contributions}

The theorems in this paper divide into two categories: those showing that $\BRT^3_2$ satisfies the same known computability-theoretic bounds as Ramsey's theorem for pairs, and those separating $\BRT^3_2$ from $\RT^2_2$ over variations of the $\omega$-reduction. First and foremost, a classical argument shows that $\RT^2 \leq_c \BRT^3_2$ and that $\RT^2_2 \leq_c \BRT^3_{2,4}$. Therefore, all the upper bounds of $\BRT^3_2$ have to be at least as complicated as those of $\RT^2$.
\bigskip

\emph{Similar upper bounds}. Seetapun~\cite{seetapun1995strength} proved that Ramsey's theorem for pairs admits cone avoidance, and Liu~\cite{liu_cone_2010} proved that it admits constant-bound trace avoidance (defined in \Cref{sec:cone-cb-avoidance}). In his thesis, Sold\`a~\cite[Section 4.1]{solda2021calibrating} proved that $\BRT^3_2$ also satisfies cone avoidance, using a general framework defined by Patey~\cite{patey2022ramseylike}. Using the Erd\H{o}s-Moser theorem, we give an alternative proof of this fact, and we prove the following theorem:

\begin{theorem}\label[theorem]{thm:brt32-cbtrace}
$\BRT^3_2$ admits constant-bound trace avoidance.
\end{theorem}

In terms of effectiveness, Cholak, Jockusch and Slaman~\cite{cholak2001strength} proved that $\RT^2_2$ admits a weakly low basis, meaning that for every computable instance, any PA degree over~$\emptyset'$ computes the jump of a solution. We prove the same theorem for $\BRT^3_2$:

\begin{theorem}\label[theorem]{thm:brt32-weakly-low}
$\BRT^3_2$ admits a weakly low basis.    
\end{theorem}

\emph{Separations from Ramsey's theorem for pairs}. The similarity of the bounds satisfied by $\RT^2_2$ and $\BRT^3_2$ makes it complicated to separate the statements. In particular, we leave open the question whether $\BRT^3_2 \leq_\omega \RT^2_2$. Using the framework of preservations of hyperimmunities, we however prove that $\BRT^3_{2,4}$ cannot be solved by any fixed number of applications of $\RT^2_2$:

\begin{theorem}\label[theorem]{thm:brt3-not-omegak-rt22}
For every~$k \in \NN$, $\BRT^3_{2,4} \not \leq_\omega^k \RT^2_2$.
\end{theorem}

Last, using an ad-hoc construction, we prove the existence of a computable instance of $\BRT^3_{2,4}$ which defeats every computable instance of $\RT^2$, no matter the number of colors.

\begin{theorem}\label[theorem]{thm:brt3-not-omega-rt2}
$\BRT^3_{2,4} \not \leq_c \RT^2$.
\end{theorem}

\section{$\BRT^3_2$ and the Erd\H{o}s-Moser theorem}\label[section]{sec:cone-cb-avoidance}

As explained in \Cref{sec:stability-cohesiveness}, the computability-theoretic analysis of $\BRT^3_2$ is closely related to the one of $\BRT^2_2$ for $\Delta^0_2$ instances. Le Houérou and Patey~\cite{houerou2025reverse} and Sold\`a~\cite[Section 4.1]{solda2021calibrating} independently proved that $\BRT^2_2$ follows from the Erd\H{o}s-Moser theorem over~$\RCA_0$. We use this relation to derive a framework for translating upper bounds from the Erd\H{o}s-Moser theorem to $\BRT^3_2$. A \emph{weakness property} is a collection of sets~$\W \subseteq \P(\NN)$ which is downward-closed under Turing reducibility.

\begin{definition}\label[definition]{def:preservation-weakness}
A problem~$\Psf$ \emph{preserves} a weakness property~$\W$ if for every set~$Z \in \W$ and every $Z$-computable $\Psf$-instance~$X$, there is a $\Psf$-solution~$Y$ to~$X$ such that $Y \oplus Z \in \W$.
Moreover, $\Psf$ \emph{strongly preserves} $\W$ if the restriction to $Z$-computable $\Psf$-instances is replaced by arbitrary $\Psf$-instances.
\end{definition}

Many basis theorems studied in reverse mathematics can be formulated as preservations for families of weakness properties.

\begin{example}
A problem~$\Psf$ admits \emph{cone avoidance} if for every set~$Z$, every set~$C \not \leq_T Z$ and every~$Z$-computable $\Psf$-instance~$X$, there is a $\Psf$-solution~$Y$ to~$X$ such that $C \not \leq_T Y \oplus Z$.
Equivalently, $\Psf$ admits cone avoidance if it preserves $\W_C = \{ Z : C \not \leq_T Z \}$ for every set~$C$.
\end{example}

The Erd\H{o}s-Moser theorem~\cite{erdhos_representation_1964} is a statement from graph theory about finding transitive sub-tournaments. A tournament is an oriented complete graph, and it is transitive if it contains no cycle. In our situation, it is more convenient to formulate it in terms of transitive colorings of pairs. Given a coloring $f : [\NN]^2 \to 2$, a set $H \subseteq \NN$ is \emph{$f$-transitive} if for every~$x, y, z \in H$ with $x < y < z$, if $f(x, y) = f(y, z)$, then $f(x, y) = f(x, z)$.

\begin{statement}[Erd\H{o}s-Moser theorem]
$\EM$ is the statement \qt{Every coloring $f : [\NN]^2 \to 2$ has an infinite $f$-transitive subset.}
\end{statement}

The statement $\EM$ was introduced in reverse mathematics by Bovykin and Weiermann~\cite{bovykin2017strength} to provide an alternative decomposition of Ramsey's theorem for pairs. Lerman, Solomon and Towsner~\cite{manuel_lerman_separating_2013} proved that $\EM$ is strictly weaker than~$\RT^2_2$ over~$\RCA_0$.

The following lemma is an adaptation of the fact that $\RT^1_2$ is strongly computably reducible to $\EM$ and that strong preservation of a weakness property is downward-closed under strong computable reducibility.

\begin{lemma}\label[lemma]{lem:em-to-rt1-strongly-preserves}
If $\EM$ strongly preserves a weakness property~$\W$, then so does $\RT^1$.
\end{lemma}
\begin{proof}
It suffices to prove that $\RT^1_2$ strongly preserves~$\W$.
Let $Z \in \W$, and let $g : \NN \to 2$ be an instance of~$\RT^1_2$.
Define $f : [\NN]^2 \to 2$ by $f(x, y) = 1$ if and only if $g(x) \neq g(y)$.
Since $\EM$ strongly preserves~$\W$, there is an infinite $f$-transitive set~$H$ such that $H \oplus Z \in \W$.
We claim that $H$ is, up to finite changes, $g$-homogeneous. Indeed, if there are some~$x, y, z \in H$
such that $x < y < z$ and $g(x) \neq g(y)$ and $g(y) \neq g(z)$, then $f(x, y) = f(y, z) = 1$ but $f(x, z) = 0$,
contradicting the $f$-transitivity of~$H$.
\end{proof}

The following proposition essentially follows the proof by Le Houérou and Patey~\cite{houerou2025reverse} and Sold\`a~\cite[Theorem 4.1.10]{solda2021calibrating} that $\BRT^2_2$ can be solved by an application of $\EM$ followed by an application of $\RT^1$.

\begin{proposition}\label[proposition]{prop:em-brt22-strong}
If $\EM$ strongly preserves a weakness property~$\W$, then so does $\BRT^2_2$.
\end{proposition}
\begin{proof}
Let $Z \in \W$, and let $f : [\NN]^2 \to 2$ be an instance of $\BRT^2_{2,\ell}$ for some $\ell$.
Since $\EM$ strongly preserves~$\W$, there is an infinite $f$-transitive set~$S = \{ x_0 < x_1 < \cdots \}$ such that $S \oplus Z \in \W$.
Let $g : \NN \to \ell$ be defined as follows: 
given $a \in \NN$, $g(a) = s$ where $a_0 < \cdots < a_s (=a)$ is the longest sequence such that $\{ x_{a_t} : t \leq s \}$ is $f$-homogeneous for color~1. By \Cref{lem:em-to-rt1-strongly-preserves}, $\RT^1$ strongly preserves~$\W$, so there is an infinite $g$-homogeneous set~$H \subseteq \NN$ such that $H \oplus S \oplus Z \in \W$.
Note that $\hat H = \{ x_a : a \in H \}$ is $f$-homogeneous for color~0 and $\hat H \oplus Z \leq_T H \oplus S \oplus Z$. Since $\W$ is downward-closed under Turing reducibility, $\hat H \oplus Z \in \W$.
\end{proof}

\begin{remark}
In the proof of \Cref{prop:em-brt22-strong}, we use the bijection between $S$ and $\NN$ to translate a coloring of $S$ to a coloring of~$\NN$, apply the notion of preservation relativized to~$S$, and translate back the solution over~$\NN$ to a solution over~$S$. In the remainder of this article, we shall simplify the proofs by directly considering colorings of infinite sets, using transparently the translation between the domain and~$\NN$. 
\end{remark}

\begin{statement}[Cohesiveness]
$\COH$ is the statement \qt{Every infinite sequence of sets admits an infinite cohesive set.}
\end{statement}

As mentioned in \Cref{sec:stability-cohesiveness}, the cohesiveness principle can be seen as a bridge between computable colorings of $[\NN]^{n+1}$ and $\Delta^0_2$ colorings of $[\NN]^n$, thanks to \Cref{lem:coh-makes-stable}.

\begin{corollary}\label[corollary]{cor:em-brt32}
If $\EM$ strongly preserves a weakness property~$\W$ and $\COH$ preserves~$\W$,
then $\BRT^3_2$ preserves~$\W$.
\end{corollary}
\begin{proof}
Let $Z \in \W$ and $f : [\NN]^3 \to 2$ be a $Z$-computable instance of~$\BRT^3_2$.
By \Cref{lem:coh-makes-stable}, there is a uniformly $Z$-computable sequence of sets $\vec{R}$ such that for every $\vec{R}$-cohesive set~$C$, $f$ is stable on~$C$. Since $\COH$ preserves~$\W$, there is an $\vec{R}$-cohesive set~$C = \{ x_0 < x_1 < \cdots \}$
such that $C \oplus Z \in \W$. Let $g : [\NN]^2 \to 2$ be defined by $g(a, b) = \lim_{z \in C} f(x_a, x_b, z)$. Note that $g$ is the limit coloring of $f \uh [C]^3$. In particular, $g$ is an instance of $\BRT^2_2$ on the domain~$C$. Since $\EM$ strongly preserves~$\W$, then so does $\BRT^2_2$ by \Cref{prop:em-brt22-strong}. Thus there is an infinite $g$-homogeneous set~$H \subseteq C$ for color~0 such that $H \oplus C \oplus Z \in \W$, so $H \oplus Z \in \W$ by the downward-closure of~$\W$ under Turing reducibility.
\end{proof}

\subsection{Cone avoidance}

Cone avoidance is an essential tool to separate a statement from Arithmetic Comprehension ($\ACA$). Indeed, the halting set is $\Sigma^0_1$-definable, hence belongs to every $\omega$-model of $\RCA_0 + \ACA$. Jockusch and Soare~\cite{jockusch_pi0_1_1972} proved that weak K\"onig's lemma ($\WKL$) admits cone avoidance, and Seetapun~\cite{seetapun1995strength} proved that Ramsey's theorem for pairs admits cone avoidance. 
Jockusch and Stephan~\cite{jockusch_cohesive_1993} proved that $\COH$ admits cone avoidance, and Patey and Wong independently proved that $\EM$ admits strong cone avoidance (see~\cite{levy2025weakness}).
Patey~\cite{patey2022ramseylike} introduced a general family of statements called \emph{Ramsey-like}, and he characterized which ones admit cone avoidance. This was used by Sold\`a~\cite[Section 4.1]{solda2021calibrating} to prove that $\BRT^3_2$ admits cone avoidance. We reprove this theorem using our framework:

\begin{theorem}[Sold\`a]
$\BRT^3_2$ admits cone avoidance.
\end{theorem}
\begin{proof}
Immediate by \Cref{prop:em-brt22-strong} and \Cref{cor:em-brt32} and the facts that $\COH$ admits cone avoidance and $\EM$ admits strong cone avoidance.
\end{proof}

\subsection{Constant-bound trace avoidance}

Constant-bound trace avoidance was introduced by Liu~\cite{liu_cone_2010} to separate Ramsey's theorem for pairs from weak weak K\"onig's lemma ($\WWKL$). Weak weak K\"onig's lemma is the restriction of $\WKL$ to trees of positive measure.

Let $\C \subseteq \cs$ be an $F_\sigma$ class. A \emph{$k$-trace} of~$\C$ is a sequence $(F_n)_{n \in \NN}$ such that for every~$n \in \NN$, $F_n$ is a set of size~$k$ of strings of length~$n$ such that $\C \cap [F_n] \neq \emptyset$ (that is, there is some~$X \in \C$ and some~$s \in \NN$ such that $X \uh s \in F_n$). A \emph{constant-bound trace (c.b-trace)} is a $k$-trace for some $k \in \NN$.

\begin{definition}
A problem~$\Psf$ admits \emph{c.b-trace avoidance} if for every set~$Z$, every $F_\sigma$ class~$\C \subseteq \cs$ with no $Z$-computable c.b-trace and every $Z$-computable $\Psf$-instance~$X$, there is a $\Psf$-solution~$Y$ to~$X$ such that $\C$ has no $Y \oplus X$-computable c.b-trace.
\end{definition}

If we consider arbitrary $\Psf$-instances instead of $Z$-computable $\Psf$-instances, then this yields \emph{strong c.b-trace avoidance}.
Liu~\cite{liu_cone_2010} proved that $\RT^1_2$ admits strong c.b-trace avoidance and that $\COH$ admits c.b-trace avoidance, and he deduced that $\RT^2_2$ admits c.b-trace avoidance. Patey~\cite{pateycombinatorial} proved that $\EM$ admits strong c.b-trace avoidance. 

\begin{reptheorem}{thm:brt32-cbtrace}
$\BRT^2_2$ admits strong c.b-trace avoidance, and $\BRT^3_2$ admits c.b-trace avoidance.
\end{reptheorem}
\begin{proof}
Immediate by \Cref{prop:em-brt22-strong} and \Cref{cor:em-brt32} and the facts that $\COH$ admits c.b-trace avoidance (Liu~\cite{liu_cone_2010}) and $\EM$ admits strong c.b-trace avoidance (Patey~\cite{pateycombinatorial}).
\end{proof}

It follows that $\WWKL \not \leq_\omega \BRT^3_2$, hence that $\BRT^3_2$ does not imply~$\WWKL$ over~$\RCA_0$.

\section{Effective bounds of $\BRT^3_2$}

Jockusch~\cite{jockusch_ramseys_1972} studied the complexity of solutions to Ramsey's theorem in the arithmetic hierarchy. More precisely, he proved that for every~$n, k \geq 1$, every computable instance of $\RT^n_k$ admits a $\Pi^0_n$ solution and that for every~$n \geq 2$, there is a computable instance of $\RT^n_2$ with no $\Sigma^0_n$ solution. The bounded version of Ramsey's theorem shares the same arithmetic bounds shifted by one. That is, for every~$n \geq 2$, every computable instance of $\BRT^n_2$ admits a $\Pi^0_{n-1}$ solution, and for every~$n \geq 3$, there is a computable instance of $\BRT^n_2$ with no $\Sigma^0_{n-1}$ solution.

However, arithmetic bounds do not behave well with respect to iterations and constructions of $\omega$-models, in that a $\Delta^0_n$ set relative to a $\Delta^0_n$ set is not in general $\Delta^0_n$. This motivates the study of preservations of lowness notions.

\begin{definition}
A problem~$\Psf$ admits a \emph{low${}_n$ basis} if for every set~$Z$ and every $Z$-computable $\Psf$-instance~$X$, there is a $\Psf$-solution~$Y$ to~$X$ such that $(Y \oplus Z)^{(n)} \leq_T Z^{(n)}$.
\end{definition}

In particular, if a problem $\Psf$ admits a low${}_n$ basis, it admits an $\omega$-model with only low${}_n$ sets. This provides an alternative way of separating a statement from $\ACA$ over~$\RCA_0$, since the halting set is not of low${}_n$ degree for any~$n$. In some situations, a problem is \qt{close} to admitting a low${}_n$ basis, in that one can decide the $\Sigma^0_n$ properties of a solution, except that the whole construction must be carried out with a PA degree over $\emptyset^{(n)}$ instead of only a $\emptyset^{(n)}$ oracle. This yields the notion of weakly low${}_n$ basis.

\begin{definition}
A problem~$\Psf$ admits a \emph{weakly low${}_n$ basis} if for every set~$Z$, every set~$P$ of PA degree over~$Z^{(n)}$ and every $Z$-computable $\Psf$-instance~$X$, there is a $\Psf$-solution~$Y$ to~$X$ such that $(Y \oplus Z)^{(n)} \leq_T P$.
\end{definition}

The most notable example of problem admitting a weakly low basis, but no low basis, is Ramsey's theorem for pairs. Indeed, by Jockusch~\cite{jockusch_ramseys_1972}, there is a computable instance of $\RT^2_2$ with no $\Sigma^0_2$ (hence no low) solution, but Cholak, Jockusch and Slaman~\cite{cholak2001strength} proved that $\RT^2_2$ admits a weakly low basis. By the low basis theorem, if a problem admits a weakly low${}_n$ basis, then it admits a low${}_{n+1}$ basis. Since any PA degree bounds a Scott ideal, if $\Psf$ admits a weakly low${}_n$ basis, then for every set~$P$ of PA degree over~$\emptyset^{(n)}$, there is an $\omega$-model~$\M$ of $\Psf$ such that for every $Z \in \M$, $Z^{(n)} \leq_T P$. 

Wang~\cite{wang_cohesive_2014} proved that $\COH$ for $\Delta^0_2$ instances admits a weakly low${}_2$ basis, hence a low${}_3$ basis, and Levy Patey and Mimouni~\cite[Theorem 7.1]{levy2025weakness} proved that $\EM$ for $\Delta^0_n$ instances admits a weakly low${}_n$ basis for $n \geq 1$. Following the relation between the Erd\H{o}s-Moser theorem and $\BRT^3_2$ developed in~\Cref{sec:cone-cb-avoidance}, we deduce a first effective upper bound for $\BRT^3_2$.

\begin{proposition}
$\BRT^3_2$ admits a weakly low${}_2$ basis.
\end{proposition}
\begin{proof}
Let $Z$ be a set, $P$ be of PA degree over $Z''$ and $f : [\NN]^3 \to 2$ be a $Z$-computable instance of $\BRT^3_{2,k+1}$ for some~$k \in \NN$. 
Let $\vec{R}$ be the sequence of sets defined by $R_{x, y} = \{ z : f(x, y, z) = 1 \}$. By Cholak, Jockusch and Slaman~\cite{cholak2001strength}, $\COH$ admits a low${}_2$ basis, so there is an $\vec{R}$-cohesive set~$C$ such that $(C \oplus Z)'' \leq_T Z''$. In particular, $f$ is stable on $[C]^3$. Let $\hat f : [C]^2 \to 2$ be its limit coloring, that is, $\hat f(x, y) = \lim_{z \in C} f(x, y, z)$. Note that $\hat f$ is a $(C \oplus Z)'$-computable instance of~$\BRT^2_{2,k}$ and that $P$ is of PA degree over $(C \oplus Z)''$. By Levy Patey and Mimouni~\cite[Theorem 7.1]{levy2025weakness}, $\EM$ for $\Delta^0_2$ instances admits a weakly low${}_2$ basis. Therefore, there is an infinite $\hat f$-transitive set~$S \subseteq C$ such that $(S \oplus C \oplus Z)'' \leq_T P$. 

Let $g : S \to k$ be defined as follows: 
given $x \in S$, $g(x) = s$ where $x_0 < \cdots < x_s (=x)$ is the longest sequence of elements of~$S$ which is $\hat f$-homogeneous for color~1. By Cholak, Jockusch and Slaman~\cite{cholak2001strength}, $\mathsf{D}^2_2$ admits a low${}_2$ basis, so there is a $g$-homogeneous set~$H \subseteq S$ such that $(H \oplus S \oplus C \oplus Z)'' \leq_T (S \oplus C \oplus Z)'' \leq_T P$. Note that $H$ is $\hat f$-homogeneous for color~0 and hence that $H \oplus Z$ can thin $H$ to a set that is $f$-homogeneous for color~0.
\end{proof}

The remainder of this section is devoted to the proof of the following theorem:

\begin{reptheorem}{thm:brt32-weakly-low}
$\BRT^3_2$ admits a weakly low basis.
\end{reptheorem}

The proof of \Cref{thm:brt32-weakly-low} relies on \Cref{thm:delta2-cohesive-brt22-weakly-low,thm:delta2-stable-brt22-weakly-low}, proven in \Cref{sec:weakly-low-delta2-coh-brt22,sec:weakly-low-delta2-stable-brt22}, respectively.  
\Cref{thm:delta2-cohesive-brt22-weakly-low} states that for any $\Delta^0_2$ instance~$f$ of $\BRT^2_2$, any PA degree over $\emptyset'$ computes the jump of an infinite set over which~$f$ is stable.
\Cref{thm:delta2-stable-brt22-weakly-low} states that for any stable $\Delta^0_2$ instance~$f$ of $\BRT^2_2$, any PA degree over $\emptyset'$ computes the jump of an infinite $f$-homogeneous set for color~0.
We now prove \Cref{thm:brt32-weakly-low} assuming those two theorems.

\begin{proof}[Proof of \Cref{thm:brt32-weakly-low}]
Let $Z$ be a set, $P$ be of PA degree over $Z'$ and $f : [\NN]^3 \to 2$ be a $Z$-computable instance of $\BRT^3_2$.
By Scott~\cite{scott1962algebras}, there are some $P_1, P_2$ such that $P \gg P_1 \gg P_2 \gg Z'$, where $A \gg B$ means that $A$ is of PA degree over~$B$.
Let $\vec{R}$ be the sequence of sets defined by $R_{x, y} = \{ z : f(x, y, z) = 1 \}$. By Cholak, Jockusch and Slaman~\cite{cholak2001strength}, $\COH$ admits a weakly low basis, so there is an $\vec{R}$-cohesive set~$C$ such that $(C \oplus Z)' \leq_T P_2$. In particular, $f$ is stable on $[C]^3$. Let $\hat f : [C]^2 \to 2$ be its limit coloring, that is, $\hat f(x, y) = \lim_{z \in C} f(x, y, z)$. Note that $\hat f$ is a $(C \oplus Z)'$-computable instance of~$\BRT^2_2$. 

By \Cref{thm:delta2-cohesive-brt22-weakly-low}, there is an infinite set $C_1 \subseteq C$ on which $\hat f$ is stable and such that $(C_1 \oplus C \oplus Z)' \leq_T P_1$. By \Cref{thm:delta2-stable-brt22-weakly-low}, there is  an infinite $\hat f$-homogeneous subset~$H \subseteq C_1$ for color~0 such that $(H \oplus C_1 \oplus C \oplus Z)' \leq_T P$. In particular, $(H \oplus Z)' \leq_T P$, and $H \oplus Z$ can thin $H$ to a set that is $f$-homogeneous for color $0$.
\end{proof}

\subsection{Weakly low basis for $\Delta^0_2$ cohesive~$\BRT^2_2$}\label[section]{sec:weakly-low-delta2-coh-brt22}

The goal of this section is to prove the following theorem:

\begin{theorem}\label[theorem]{thm:delta2-cohesive-brt22-weakly-low}
Let $f : [\NN]^2 \to 2$ be a $\Delta^0_2$ instance of $\BRT^2_2$ and $P$ be of PA degree over~$\emptyset'$.
There is an infinite set~$G$ with $G' \leq_T P$ on which $f$ is stable and the limit is aways~0: $\lim_{y \in G} f(x,y) = 0$ for all $x \in G$.
\end{theorem}
\begin{proof}
Fix~$f$, and let $k \in \NN$ be such that there is no $f$-homogeneous set for color~1 of size~$k$.
\bigskip

\noindent
\textbf{$f$-tree.}
An \emph{$f$-tree} is a non-empty tree $T \subseteq \NN^{<\NN}$ such that for every~$\sigma \in T$, 
$\sigma$ is an increasing sequence that is $f$-homogeneous for color~1.  That is, for every $i < |\sigma|-1$, $\sigma(i) < \sigma(i+1)$, and for every $i < j < |\sigma|$, $f(\sigma(i), \sigma(j)) = 1$. 
Given $\sigma \in T$, we write $F(T, \sigma)$ for the set $\{ x \in \NN : \sigma \cdot x \in T \}$. Note that $F(T, \sigma)$ is not necessarily $f$-homogeneous for color~0. Requiring $T$ to be non-empty is equivalent to saying that $T$ contains the node $\epsilon$ of length~0. Also note that since $f$ is an instance of $\BRT^2_{2,k}$, every $f$-tree has depth at most~$k$. We write $T^*$ for the sub-tree of~$T$ of depth at most~$k-1$, that is, $T^* = \{ \sigma \in T : |\sigma| < k-1\}$. If $T$ is an infinite $f$-tree, then there is a node $\sigma \in T^*$ such that $F(T, \sigma)$ is infinite. We also write $D(T)$ for the union of the range of every node, that is, $D(T) = \bigcup_{\sigma \in T} F(T, \sigma)$.
\bigskip

\noindent
\textbf{Notion of forcing.}
Consider the notion of forcing whose conditions are pairs $(T, X)$, where
\begin{itemize}
    \item[(1)] $T \subseteq \NN^{<\NN}$ is a finite $f$-tree;
    \item[(2)] $X \subseteq \NN$ is an infinite set such that $D(T) < X$;
    \item[(3)] $X$ is low.
\end{itemize}
A condition $q = (S, Y)$ \emph{extends} $p = (T, X)$ if $T \subseteq S$, $Y \subseteq X$, $D(S) \setminus D(T) \subseteq X$ and finally $f(x,y) = 0$ whenever $\sigma \in T^*$, $x \in F(T, \sigma)$ and $y \in F(S, \sigma) \setminus F(T, \sigma)$.
Think of a condition $p = (T, X)$ as simultaneous Mathias conditions $p^{[\sigma]} = (F(T, \sigma), X)$ for every~$\sigma \in T^*$. In particular, if $q = (S, Y)$ extends $p = (T, X)$, then for every~$\sigma \in T^*$, $q^{[\sigma]}$ Mathias extends $p^{[\sigma]}$. A \emph{code} of a condition $(T, X)$ is a pair $(T, e)$ such that $e$ is a lowness index for~$X$, that is, $\Phi_e^{\emptyset'} = X'$. 

Any decreasing sequence of conditions $p_0 \geq p_1 \geq \cdots$ with $p_s = (T_s, X_s)$ induces an $f$-tree $T = \bigcup_s T_s$, and every $\sigma \in T^*$ induces a set $G_\sigma = F(T, \sigma) = \bigcup_s F(T_s, \sigma)$. As mentioned, if $T$ is infinite, then there is some~$\sigma \in T^*$ such that $G_\sigma$ is infinite.
Given $s \in \NN$ and $\sigma \in T_s^*$, $F(T_s, \sigma)$ is not necessarily $f$-homogeneous for color~0, and so neither is $G_\sigma$, but by the definition of condition extension, for every $x \in F(T_s, \sigma)$ and $y \in G_\sigma \setminus F(T_s, \sigma)$, $f(x, y) = 0$. It follows that every element in~$G_\sigma$ will have limit color~0 in~$G_\sigma$.
\bigskip

\noindent
\textbf{Forcing $\Sigma^0_1$ and $\Pi^0_1$ formulas.}
Let $p = (T, X)$ be a condition, $\varphi(G)$ be a $\Sigma^0_1$-formula and $\sigma \in T^*$. We shall always assume that $\varphi(G)$ is in normal form and induces a monotone formula $\varphi(\rho)$ on strings such that $\varphi(G)$ holds if and only if there is some~$\rho \prec G$ such that $\varphi(\rho)$ holds.
\begin{itemize}
    \item[(1)] $p \Vdash \varphi(G_\sigma)$ if $\varphi(F(T, \sigma))$ holds.
    \item[(2)] $p \Vdash \neg \varphi(G_\sigma)$ if for every finite set $H \subseteq X$, $\varphi(F(T, \sigma) \cup H)$ does not hold.
\end{itemize}
We say that $p$ \emph{decides} $\varphi(G_\sigma)$ if $p \Vdash \varphi(G_\sigma)$ or $p \Vdash \neg\varphi(G_\sigma)$. Note that the forcing relation for $\Pi^0_1$-formulas is stronger than the semantic one and yields the following lemma.

\begin{lemma}
Let $p_0 \geq p_1 \geq \cdots$ be a decreasing sequence of conditions with $p_s = (T_s, X_s)$.
Let $T = \bigcup_s T_s$ and $\sigma \in T^*$. If for every $\Sigma^0_1$-formula~$\varphi(G)$ there is some~$s$ such that $p_s$ decides $\varphi(G_\sigma)$, then $G_\sigma$ is infinite.
\end{lemma}
\begin{proof}
For every~$n \in \NN$, we can define a $\Sigma^0_1$-formula $\varphi_n(G) \equiv \exists x > n, x \in G$. There is no condition $p$ such that $p \Vdash \neg \varphi_n(G_\sigma)$, so if some $p_s$ decides $\varphi_n(G_\sigma)$, then $p_s \Vdash \varphi(G_\sigma)$, and unfolding the definition, there is some~$x > n$ such that $x \in F(T_s, \sigma)$.
Since $G_\sigma = \bigcup_s F(T_s, \sigma)$, it follows that $G_\sigma$ is infinite.
\end{proof}
\bigskip

\noindent
\textbf{Forcing question.}
Fix an enumeration $\varphi_0, \varphi_1, \dots$ of all $\Sigma^0_1$-formulas.
Since each condition represents a tree of Mathias conditions, we define a disjunctive forcing question.

\begin{definition}
Let $p = (T, X)$ be a condition and $g : T^* \to \NN$ be a function.
Let $p \qvdash \bigvee_{\sigma \in T^*} \varphi_{g(\sigma)}(G_\sigma)$ hold if for every function $h : X \to T^*$,
there is some $h$-homogeneous set~$H \subseteq X$ for some color~$\sigma \in T^*$ such that $\varphi_{g(\sigma)}(F(T, \sigma) \cup H)$ holds.
\end{definition}

Note that by compactness, the forcing question is $\Sigma^0_1(X)$ uniformly in the parameters of the formula.

\begin{lemma}\label[lemma]{lem:delta2-cohesive-brt22-weakly-low-forcing-question-spec}
Let $p = (T, X)$ be a condition and $g : T^* \to \NN$ be a function.
\begin{itemize}
    \item[(i)] If $p \qvdash \bigvee_{\sigma \in T^*} \varphi_{g(\sigma)}(G_\sigma)$, then there is an extension $q \leq p$ and some~$\sigma \in T^*$ such that $q \Vdash \varphi_{g(\sigma)}(G_\sigma)$.
    \item[(ii)] If $p \nqvdash \bigvee_{\sigma \in T^*} \varphi_{g(\sigma)}(G_\sigma)$, then there is an extension $q \leq p$ and some~$\sigma \in T^*$ such that $q \Vdash \neg \varphi_{g(\sigma)}(G_\sigma)$.
\end{itemize}
Moreover, a code for~$q$ and the case in which we are can be $P$-computed uniformly in a code for~$p$.
\end{lemma}
\begin{proof}\ 
\begin{itemize}
    \item[(i)] Suppose $p \qvdash \bigvee_{\sigma \in T^*} \varphi_{g(\sigma)}(G_\sigma)$ holds. Then by compactness, there is a finite set $E \subseteq X$ such that for every function $h : E \to T^*$, there is some $h$-homogeneous set~$H \subseteq E$ for some color~$\sigma \in T^*$ such that $\varphi_{g(\sigma)}(F(T, \sigma) \cup H)$ holds. 
    Let $h : E \to T^*$ be such that $h(y)$ is a node of maximal length~$\sigma \in T^*$ for which $\sigma \cup \{y\}$ is $f$-homogeneous for color~1. 
    
    We claim that for every~$y \in E$ and $x \in F(T, h(y))$, $f(x, y) = 0$. Fix~$y$ and~$x$. By definition of an $f$-tree, $h(y) \cup \{x\}$ is $f$-homogeneous for color~1, and by choice of $h(y)$, so is $h(y) \cup \{y\}$. It follows that if $f(x, y) = 1$, then $h(y) \cup \{x, y\}$ is $f$-homogeneous for color~1. Since $f$ is an instance of $\BRT^2_{2,k}$, it has size less than~$k$, so $|h(y)| < k-2$, but then $\sigma = h(y) \cdot x \in T^*$ is such that $\sigma \cup \{y\}$ is $f$-homogeneous for color~1, contradicting the maximality of $h(y)$. Thus $f(x, y) = 0$.

    Let $H \subseteq E$ be an $h$-homogeneous set for some color~$\sigma \in T^*$ such that $\varphi_{g(\sigma)}(F(T, \sigma) \cup H)$ holds. Let $S = T \cup \{ \sigma \cdot y : y \in H \}$. Then $q = (S, X \setminus [0, \max H])$ is an extension of~$p$ such that $q \Vdash \varphi_{g(\sigma)}(G_\sigma)$. Moreover, a code for~$q$ can be $\emptyset'$-computed from a code of~$p$.
    
    \item[(ii)] Suppose $p \nqvdash \bigvee_{\sigma \in T^*} \varphi_{g(\sigma)}(G_\sigma)$. Let $\Q$ be the $\Pi^0_1(X)$ class of all $h : X \to T^*$ such that for every~$\sigma \in T^*$ and every~$h$-homogeneous set~$H \subseteq X$, $\varphi_{g(\sigma)}(F(T, \sigma) \cup H)$ does not hold. By assumption, $\Q \neq \emptyset$. By the uniform low basis theorem, we can $\emptyset'$-compute a lowness index of $h \oplus X$ for some~$h \in \Q$. Moreover, one can $P$-compute a lowness index of an infinite $h$-homogeneous subset~$Y \subseteq X$ for some color~$\sigma$. The condition $q = (T, Y)$ is an extension of~$p$ such that $q \Vdash \neg \varphi_{g(\sigma)}(G_\sigma)$.
\end{itemize}
\end{proof}
\bigskip

\noindent
\textbf{Construction.}
Given a condition $p = (T, X)$, we let $g_p : T^* \to \NN$ be such that for each~$\sigma \in T^*$, $g_p(\sigma)$ is the least~$e \in \NN$ such that $p$ does not decide $\varphi_e(G_\sigma)$. Using \Cref{lem:delta2-cohesive-brt22-weakly-low-forcing-question-spec}, build a $P$-computable sequence of codes corresponding to a decreasing sequence of conditions $p_0 \geq p_1 \geq \cdots$ with $p_s = (T_s, X_s)$, such that for every~$s \in \NN$, there is some~$\sigma_s \in T_s^*$ such that $p_{s+1}$ decides $\varphi_{g_{p_s}(\sigma_s)}(G_{\sigma_s})$. Let $T  = \bigcup_s T_s$. 

We claim that $T$ is infinite. Indeed, if $T$ is finite, then by the pigeonhole principle, there is some~$\sigma \in T^*$ such that $\sigma = \sigma_s$ for infinitely many~$s$. By the definition of $g_p$, every $\Sigma^0_1(G_\sigma)$-formula is decided.  Thus $G_\sigma$ is infinite, so $T$ is infinite. 

Let $\sigma \in T^*$ be such that $G_\sigma$ is infinite. Note that an extension increases the size of~$G_\sigma$ at stage~$s$ only when $\sigma = \sigma_s$, so there are infinitely many~$s$ such that $\sigma = \sigma_s$. Again by the definition of~$g_p$, every $\Sigma^0_1(G_\sigma)$-formula is decided, so $G_\sigma' \leq_T P$. By the definition of condition extension, every element of~$G_\sigma$ has limit~0 in $G_\sigma$, so $f$ is stable on~$G_\sigma$.
This completes the proof of \Cref{thm:delta2-cohesive-brt22-weakly-low}.
\end{proof}

\subsection{Weakly low basis for $\Delta^0_2$ stable $\BRT^2_2$}\label[section]{sec:weakly-low-delta2-stable-brt22}

The goal of this section is to prove the following theorem:

\begin{theorem}\label[theorem]{thm:delta2-stable-brt22-weakly-low}
Let $f : [\NN]^2 \to 2$ be a $\Delta^0_2$ stable instance of $\BRT^2_2$ and $P$ be of PA degree over~$\emptyset'$.
There is an infinite $f$-homogeneous set~$G$ for color~0 such that $G' \leq_T P$.
\end{theorem}
\begin{proof}
Fix~$f$ and~$P$. If there exists infinitely many~$x$ such that $\lim_y f(x, y) = 1$, then there exists an infinite $f$-homogeneous set for color~1, contradicting the hypothesis. Therefore, by removing finitely many integers, we can assume that for every~$x \in \NN$, $\lim_y f(x, y) = 0$. Let $k$ be such that there is no $f$-homogeneous set for color~1 of size~$k$.
\bigskip

\noindent
\textbf{Notion of forcing.}
Consider the notion of forcing whose conditions are tuples $(\vec{g}, \sigma, X)$, where
\begin{itemize}
    \item[(1)] $(\sigma, X)$ is a Mathias condition;
    \item[(2)] $\vec{g} = g_0, \dots, g_{t-1}$ is a finite sequence of instances of $\BRT^2_2$ for some~$t \in \NN$;
    \item[(3)] for every~$g \in \vec{g}$ and every~$y \in X$, $\sigma \cup \{y\}$ is $g$-homogeneous for color~0;
    \item[(4)] $\sigma$ is $f$-homogeneous for color~0;
    \item[(5)] $\vec{g} \oplus X$ is low.
\end{itemize}
A condition $q = (\vec{h}, \tau, Y)$ \emph{extends} $p = (\vec{g}, \sigma, X)$ if $(\tau, Y)$ Mathias-extends $(\sigma, X)$ and $\vec{g}$ is a sub-sequence of~$\vec{h}$. A \emph{code} of a condition $(\vec{g}, \sigma, X)$ is a pair $(\sigma, e)$ such that $e$ is a lowness index for~$\vec{g} \oplus X$, that is, $\Phi_e^{\emptyset'} = (\vec{g} \oplus X)'$. Any decreasing sequence of conditions $p_0 \geq p_1 \geq \cdots$ with $p_s = (\vec{g}_s, \sigma_s, X_s)$ induces a set~$G = \bigcup_s \sigma_s$.
\bigskip

\noindent
\textbf{Rank.}
Given an instance $g : [\NN]^2 \to 2$ of $\BRT^2_2$ and an infinite set~$X \subseteq \NN$,
let $\rank(g, X)$ be $\omega^n$ for the largest~$n$ such that there is a $g$-homogeneous subset of~$X$ for color~1 of size~$n$. Given $\alpha = \omega^n \cdot a_n + \dots + \omega^0 \cdot a_0$ and $\beta = \omega^n \cdot b_n + \dots + \omega^0 \cdot b_0$, we write $\alpha \uplus \beta$ for the natural sum of ordinals, defined by $\alpha \uplus \beta = \omega^n \cdot (a_n + b_n) + \dots + \omega^0 \cdot (a_0 + b_0)$. The \emph{rank} 
$\rank(p)$ of a condition $p = (\vec{g}, \sigma, X)$ is $\omega^\ell \uplus \biguplus_{g \in \vec{g}} \rank(g, X)$, where $\ell = \card \{ y \in X : \exists x \in \sigma, f(x, y) = 1 \}$. Note that $\rank(p)$ is not necessarily $P$-computable from a code of~$p$ because of~$\ell$. Also note that the rank might increase or decrease by condition extension. However, if $p = (\vec{g}, \sigma, X)$ and $q = (\vec{g}, \sigma, Y)$ are two conditions with $Y \subseteq X$, then $\rank(q) \leq \rank(p)$. 
\bigskip

\noindent
\textbf{Forcing infinity.}
We now prove that for every~$n \in \NN$, the set of conditions forcing $G$ to have an element greater than~$n$ is dense and that this density can be $P$-effectively witnessed. The search for an extension increasing the size of the stem can require multiple attempts, but each failure decreases the rank of the current condition, so we must eventually succeed.

\begin{lemma}\label[lemma]{lem:delta2-stable-brt22-weakly-low-forcing-infinity}
Let $p = (\vec{g}, \sigma, X)$ be a condition and $y \in X$. There is an extension $q = (\vec{g}, \tau, Y)$ such that either $y \in \tau$ or $\rank(q) < \rank(p)$. Moreover, a code for~$q$ can be $P$-computed uniformly in a code for~$p$.
\end{lemma}
\begin{proof}
Let~$y \in X$. If there is some~$x \in \sigma$ such that $f(x, y) = 1$, then $q = (\vec{g}, \sigma, X \setminus \{x\})$ is an extension such that $\rank(q) < \rank(p)$. Moreover, one can $\emptyset'$-decide whether we are in this case or not and $\emptyset'$-compute a code for~$q$ from a code for~$p$. From now on, suppose that $\sigma \cup \{y\}$ is $f$-homogeneous for color~0.

By multiple applications of $\RT^1_2$, define (uniformly in~$P$) an infinite $\vec{g} \oplus X$-computable subset~$Y \subseteq X \setminus [0,y]$ such that for every~$g \in \vec{g}$, either for every~$z \in Y$, $g(y, z) = 0$, or for every~$z \in Y$, $g(y, z) = 1$. If there is some~$g \in \vec{g}$ such that for every~$z \in Y$, $g(y, z) = 1$, then $\rank(g,  Y) < \rank(g, X)$, so letting $q = (\vec{g}, \sigma, Y)$, $\rank(q) < \rank(p)$. Suppose now that for every~$g \in \vec{g}$ and every~$z \in Y$, $g(y, z) = 0$.
The condition $q = (\vec{g}, \sigma \cup \{y\}, Y)$ is a valid extension of~$p$, and one can $P$-computably find a code for~$q$ from a code for~$p$.
\end{proof}
\bigskip

\noindent
\textbf{Forcing question.}
We are going to define a pseudo forcing question $p \qvdash \varphi(G)$ between a condition~$p$ and a $\Sigma^0_1$-formula $\varphi(G)$ such that if $p \nqvdash \varphi(G)$ there is an extension~$q \leq p$ forcing $\neg \varphi(G)$; but if $p \qvdash \varphi(G)$ there is an extension $q \leq p$ such that either $\rank(q) < \rank(p)$ or such that $q$ forces $\varphi(G)$. Moreover, such an extension, and in which case we are, can be uniformly computed from~$P$, so it suffices to ask iteratively the question until we find an extension deciding the formula.
Given $\sigma$ and $X$, let $\P(\sigma, X)$ be the $\Pi^0_1(X)$ class of all $\BRT^2_{2,k}$-instances~$h : [\NN]^2 \to 2$ such that for every~$y \in X$, $\sigma \cup \{y\}$ is $h$-homogeneous for color~0.

\begin{definition}
Given a condition $p = (\vec{g}, \sigma, X)$ and a $\Sigma^0_1$-formula $\varphi(G)$, let $p \qvdash \varphi(G)$ hold if for every $h \in \P(\sigma, X)$ there is a finite set~$\rho \subseteq X$ which is $h$-homogeneous for color~0, $\vec{g}$-homogeneous for color~0 and such that $\varphi(\sigma \cup \rho)$ holds.
\end{definition}

Note that by compactness, the pseudo forcing question is $\Sigma^0_1(\vec{g} \oplus X)$ uniformly in the parameters of the formula.

\begin{lemma}\label[lemma]{lem:delta2-stable-brt22-weakly-low-forcing-question-spec}
Let $p$ be a condition and $\varphi(G)$ be a $\Sigma^0_1$-formula.
\begin{itemize}
    \item[(i)] If $p \qvdash \varphi(G)$, then there is an extension $q \leq p$ such that either $\rank(q) < \rank(p)$ or $q$ forces $\varphi(G)$.
    \item[(ii)] If $p \nqvdash \varphi(G)$, then there is an extension~$q \leq p$ forcing $\neg \varphi(G)$.
\end{itemize}
Moreover, a code for~$q$ and the case in which we are can be $P$-computed uniformly in a code of~$p$.
\end{lemma}
\begin{proof}
Say $p = (\vec{g}, \sigma, X)$.
\begin{itemize}
    \item[(i)] Suppose $p \qvdash \varphi(G)$ holds. By compactness, there is some finite set~$F \subseteq X$ such that for every $\BRT^2_{2,k}$-instance~$h : [F]^2 \to 2$, there is a finite set $\rho \subseteq F$ which is $h$-homogeneous for color~0, $\vec{g}$-homogeneous for color~0 and such that $\varphi(\sigma \cup \rho)$ holds.  One can $\vec{g} \oplus X$-computably find such a set~$F$.  Here, notice that if $F \subseteq X$ is finite and $h : [F]^2 \to 2$ has no homogeneous set for color 1 of size $k$, then $h$ extends to an element of $\P(\sigma, X)$.  
    
    Given such a set~$F$, $\emptyset'$-decide whether for every~$x \in \sigma$ and $y \in F$, $f(x, y) = 0$. If it is not the case, then the condition $q = (\vec{g}, \sigma, X \setminus F)$ is an extension of~$p$ of smaller rank. So from now on, suppose that for every~$x \in \sigma$ and $y \in F$, $f(x, y) = 0$.

    By multiple applications of $\RT^1_2$, define (uniformly in~$P$) an infinite $\vec{g} \oplus X$-computable subset~$Y \subseteq X \setminus [0, \max F]$ such that for every~$g \in \vec{g}$ and every~$x \in F$, either for every~$y \in Y$, $g(x, y) = 0$ or for every~$y \in Y$, $g(x, y) = 1$. If there is some~$g \in \vec{g}$ and some~$x \in F$ such that for every~$y \in Y$, $g(x, y) = 1$, then $\rank(g,  Y) < \rank(g, X)$, so letting $q = (\vec{g}, \sigma, Y)$, $\rank(q) < \rank(p)$. So from now on, suppose that for every~$g \in \vec{g}$, every~$x \in F$ and every~$y \in Y$, $g(x, y) = 0$.
    
    Letting $h = f \uh [F]^2$, let $\rho \subseteq X$ be $f$-homogeneous for color~0, $\vec{g}$-homogeneous for color~0 and such that $\varphi(\sigma \cup \rho)$ holds. In particular, $\sigma \cup \rho$ is $f$-homogeneous for color~0,
    and for every~$y \in Y$ and every~$g \in \vec{g}$, $\sigma \cup \rho \cup \{y\}$ is $g$-homogeneous for color~0.
    So $q = (\vec{g}, \sigma \cup \rho, Y)$ is a valid extension of~$p$ forcing $\varphi(G)$. 
    
    In terms of uniformity, in first case, a code for~$q$ can be $\emptyset'$-computed from a code for~$p$, while in the two other cases, a code for~$q$ is $P$-computed from a code for~$p$. We can furthermore $P$-decide in which of the three cases we are.

    \item[(ii)] Suppose $q \nqvdash \varphi(G)$ holds. Let $\Q$ be the $\Pi^0_1(\vec{g} \oplus X)$ class of all~$h \in \P(\sigma, X)$ such that for every finite set $\rho \subseteq X$ which is $h$-homogeneous for color~0 and $\vec{g}$-homogeneous for color~0, $\varphi(\sigma \cup \rho)$ does not hold. By assumption, $\Q \neq \emptyset$, and by the uniform low basis theorem, one can $\emptyset'$-compute a lowness index for $h \oplus \vec{g} \oplus X$ for some~$h \in \Q$. The condition $q = (\vec{g} \cup \{h\}, \sigma, X)$ is an extension of~$p$ forcing $\neg \varphi(G)$. Moreover, a code for $q$ can be $\emptyset'$-computed from a code for~$p$.
\end{itemize}
\end{proof}
\bigskip

\noindent
\textbf{Construction.} Build a $P$-computable sequence of codes corresponding to a decreasing sequence of conditions $p_0 \geq p_1 \geq \cdots$ with $p_s = (\vec{g}_s, \sigma_s, X_s)$ such that for every~$s \in \NN$,
\begin{itemize}
    \item[(1)] If $s = 2e$, $\card \sigma_{s+1} > \card \sigma_s$;
    \item[(2)] If $s = 2e+1$, either $p_{s+1}$ forces $\Phi^G_e(e)\converge$ or $p_{s+1}$ forces $\Phi^G_e(e)\diverge$.
\end{itemize}
Item (1) can be ensured by successive applications of \Cref{lem:delta2-stable-brt22-weakly-low-forcing-infinity} and Item (2) by successive applications of \Cref{lem:delta2-stable-brt22-weakly-low-forcing-question-spec}.
Let $G$ be the set corresponding to the decreasing sequence. By Item (1), $G$ is infinite, by the definition of a condition, $G$ is $f$-homogeneous for color~0 and by Item (2), $P \geq_T G'$.
This completes the proof of \Cref{thm:delta2-stable-brt22-weakly-low}.
\end{proof}

\section{Preserving multiple hyperimmunities}

We now turn to proofs of separations for $\BRT^3_2$ from $\RT^2_2$. Although the separation over $\omega$-models remains open, we prove that any fixed number of applications of $\RT^2_2$ is not sufficient to solve $\BRT^3_2$. More precisely, this section is devoted to the proof of the following theorem:

\begin{reptheorem}{thm:brt3-not-omegak-rt22}
For every~$k \in \NN$, $\BRT^3_{2,4} \not \leq_\omega^k \RT^2_2$.
\end{reptheorem}

Note that if $\BRT^3_2 \leq_\omega \RT^2_2$, then \Cref{thm:brt3-not-omegak-rt22} cannot be obtained by proving that $\RT^2_2$ preserves some weakness property in the sense of \Cref{def:preservation-weakness}. Indeed, such a notion of preservation can be iterated and would yield a separation with respect to $\omega$-reducibility.  Also note that \Cref{thm:brt3-not-omegak-rt22} implies that $\BRT^3_{2,4} \not \leq_\omega^k \RT^2_\ell$ for every~$k, \ell \in \NN$, as a simple color-merging argument shows that $\RT^2_\ell \leq_\omega^{\ell-1} \RT^2_2$.

We shall use the framework of preservation of hyperimmunities, which was already shown to be fruitful for separating $\RT^n_k$ from $\RT^n_{k+1}$ with respect to computable reducibility, while those two statements are equivalent over $\RCA_0$.

A function $h : \NN \to \NN$ is \emph{$Z$-hyperimmune} if it is not dominated by any $Z$-computable function. When $Z = \emptyset$, we shall simply say \emph{hyperimmune}. As many computability-theoretic notions, this induces multiple notions of preservation, as follows:

\begin{definition}Fix $k \leq \omega$.
A problem~$\Psf$ admits \emph{preservation of $k$ hyperimmunities} if for every set~$Z$,
every sequence $\langle g_s : s < k \rangle$ of $Z$-hyperimmune functions and every~$Z$-computable $\Psf$-instance~$X$, there is a $\Psf$-solution~$Y$ to~$X$ such that all the $g_s$ are $(Y \oplus Z)$-hyperimmune.
\end{definition}

Downey, Greenberg, Harrison-Trainor, Patey and Turetsky~\cite{downey_relationships_2022} proved that preservation of 1 hyperimmunity is equivalent to cone avoidance, which is itself equivalent to simultaneous avoidance of countably many cones. On the other hand, preservation of $k$ hyperimmunities forms a strictly increasing hierarchy when~$k$ increases. We shall actually work with a twisted notion of preservation of $k$ hyperimmunities in which only $\ell$ of the $k$ hyperimmune functions are required to be preserved. 

\begin{definition}Fix $\ell, k \geq 1$.
A problem~$\Psf$ admits \emph{preservation of $\ell$ among $k$ hyperimmunities} if for every set~$Z$,
every sequence $\langle g_s : s < k \rangle$ of $Z$-hyperimmune functions and every~$Z$-computable $\Psf$-instance~$X$,
there is a $\Psf$-solution~$Y$ to~$X$ and some~$I \subseteq \{0, \dots, k-1\}$ of size~$\ell$ such that for each~$s \in I$, $g_s$ is $(Y \oplus Z)$-hyperimmune.
\end{definition}

This is not a notion of preservation in the sense of \Cref{def:preservation-weakness}, but it can be partially iterated as follows: if $\Psf$ admits preservation of $k_1$ among $k_0$ hyperimmunities and of $k_2$ among $k_1$ hyperimmunities, then intuitively, two consecutive applications of $\Psf$ preserve $k_2$ among $k_0$ hyperimmunities.
Patey~\cite{patey2016reverse} proved that for every~$\ell, k \geq 1$, $\RT^2_k$ admits preservation of $\ell$ among $k\cdot (\ell-1)+1$ hyperimmunities but not preservation of $\ell$ among $k \cdot (\ell-1)$ hyperimmunities.

\subsection{Negative preservations}

Fix a digraph $G = (V,E)$. A \emph{transitive triangle} is a triple $(u, v, w) \in V^3$ such that $(u, v), (v, w), (u, w) \in E$. An \emph{independent set} is a subset~$S \subseteq V$ such that $S^2 \cap E = \emptyset$.

\begin{proposition}\label[proposition]{prop:brt223-preservation-based-on-graph}
Suppose that $G = (\{0, \dots, k-1\}, E)$ is a digraph with no transitive triangle and no independent set of size~$\ell$. Then there is a $\Delta^0_2$ instance $f : [\NN]^2 \to 2$ of $\BRT^2_{2,3}$ and $k$ hyperimmune functions $(h_n)_{n < k}$ such that for every infinite set~$H$ on which $f$ is stable, at most $\ell-1$ many $h$'s are $H$-hyperimmune.
\end{proposition}
\begin{proof}
Let $g : \NN \to k$ be a $\Delta^0_2$ function such that for every~$i < k$, the (principal function of the) set $A_i = \{ x : g(x) \neq i \}$ is hyperimmune. Such a function can be obtained by the finite extension method using a $\emptyset'$-computable construction (see Patey~\cite[Theorem 8.2.1]{patey2016reverse}).
Let $f : [\NN]^2 \to 2$ be a $\Delta^0_2$ coloring defined for every~$x < y$ by $f(x, y) = 1$ if and only if $(g(x), g(y)) \in E$.

\bigskip

\noindent
\textbf{Claim 1.} $f$ is an instance of $\BRT^2_{2,3}$.
Indeed, if $x < y < z$ is such that $\{ x, y, z\}$ is $f$-homogeneous for color~1, then $(g(x), g(y)), (g(y), g(z)), (g(x), g(z)) \in E$, and therefore $G$ admits a transitive triangle, contradicting our hypothesis. This proves our claim.
\bigskip

Let $H$ be an infinite set on which $f$ is stable. By removing finitely many elements, we can assume that for every~$x \in H$, $\lim_{y \in H} f(x, y) = 0$.
Let $I \subseteq \{0, \dots, k-1\}$ be such that $A_i$ is $H$-hyperimmune for each $i \in I$.

\bigskip

\noindent
\textbf{Claim 2.} $\card I < \ell$.
It suffices to show that $I$ is an independent set. Indeed, if $\{i,j\} \in [I]^2$, then since $A_i$ and $A_j$ are $H$-hyperimmune, $\overline{A}_i \cap H$ and $\overline{A}_j \cap H$ are both infinite. In particular, letting $x \in H$ be such that $g(x) = i$ and letting $y$ be large enough to witness that $f(x, y) = 0$ and $g(y) = j$, then $(g(x), g(y)) = (i, j) \not \in E$. This proves our claim.
\end{proof}

\begin{corollary}\label[corollary]{cor:digraph-brt-hyp-non-preserve}
Suppose that $G = (\{0, \dots, k-1\}, E)$ is a digraph with no transitive triangle and no independent set of size~$\ell$.
Then $\BRT^3_{2,4}$ does not preserve $\ell$ among $k$ hyperimmunities.
\end{corollary}
\begin{proof}
Let $g : [\NN]^2 \to 2$ be the $\Delta^0_2$ instance of $\BRT^2_{2,3}$ and $(h_n)_{n < k}$ be the hyperimmune functions from \Cref{prop:brt223-preservation-based-on-graph}. By \Cref{lem:delta2-stable-brt}, there is a computable stable instance $f : [\NN]^3 \to 2$ of $\BRT^3_{2,4}$ whose limit coloring is~$g$. It follows that any infinite $f$-homogeneous set~$H$ is $g$-homogeneous and that at most $\ell-1$ many $h$'s are $H$-hyperimmune. Hence $\BRT^3_{2,4}$ does not admit preservation of $\ell$ among $k$ hyperimmunities
\end{proof}

\begin{example}
$\BRT^3_{2,4}$ does not preserve 2 among 3 hyperimmunities, as witnessed by the following digraph:
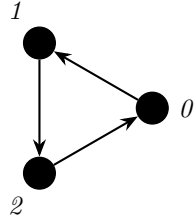
\begin{center}
\begin{tikzpicture}[>=Stealth] 
  \GraphInit[vstyle=Classic]   
  \SetGraphUnit{1}             
  \Vertices{circle}{0,1,2}

  \Edges[style={->}](0,1,2,0)
\end{tikzpicture}
\end{center}
$\BRT^3_{2,4}$ does not admit preservation of 3 among 8 hyperimmunities, as witnessed by the following digraph:
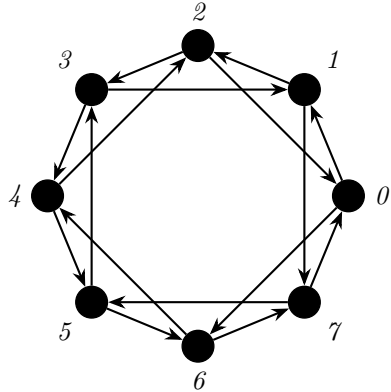
\begin{center}
\begin{tikzpicture}[>=Stealth] 
  \GraphInit[vstyle=Classic]   
  \SetGraphUnit{2}             
  \Vertices{circle}{0,1,2,3,4,5,6,7}
  
  \Edges[style={->}](0,1,2,3,4,5,6,7,0)
  \Edges[style={->}](0,6,4,2,0)
  \Edges[style={->}](7,5,3,1,7)
\end{tikzpicture}
\end{center}
\end{example}

Fix an undirected graph $G = (V, E)$. The graph~$G$ is \emph{triangle-free} if it contains no clique of size~3. An \emph{independent set} is a subset~$S \subseteq V$ such that $[S]^2 \cap E = \emptyset$.

\begin{corollary}\label[corollary]{cor:graph-brt-hyp-non-preserve}
Suppose that $G = (\{0, \dots, k-1\}, E)$ is a triangle-free undirected graph with no independent set of size~$\ell$.
Then $\BRT^3_{2,4}$ does not preserve $\ell$ among $k$ hyperimmunities.
\end{corollary}
\begin{proof}
Any orientation of the edges of $G$ yields a digraph with no transitive triangle and no independent set of size~$\ell$. Then apply \Cref{cor:digraph-brt-hyp-non-preserve}.
\end{proof}

\begin{theorem}[Erd\H{o}s~\cite{erdos1957remarks}]\label[theorem]{thm:off-diagonal-non-linear}
For every~$k \in \NN$ and every sufficiently large~$\ell \in \NN$, there is a triangle-free graph with $k\cdot (\ell-1)+1$ vertices having no independent set of size~$\ell$.
\end{theorem}
\begin{proof}
Let $R(3,\ell)$ be the \emph{off-diagonal Ramsey number}, that is, the least integer such that every triangle-free graph with $R(3,\ell)$ vertices admits an independent set of size~$\ell$. Erd\H{o}s~\cite{erdos1957remarks} proved the existence of a constant~$\epsilon > 0$ such that $R(3,\ell) > \ell^{1+\epsilon}$. For every $k \in \NN$ and every sufficiently large~$\ell \in \NN$, $R(3,\ell) > \ell^{1+\epsilon} > k\cdot (\ell-1)+1$.  Hence for every $k$ and sufficiently large $\ell$, there is a triangle-free graph with $k\cdot (\ell-1)+1$ vertices having no independent set of size~$\ell$.
\end{proof}

\begin{proposition}\label[proposition]{prop:player-rt22-preservation}
Fix~$k, \ell \geq 1$ and let $\Psf$ be a problem.
If $\Psf \leq_\omega^k \RT^2_2$, then $\Psf$ admits preservation of $\ell$ among $2^k \cdot (\ell-1)+1$ hyperimmunities.
\end{proposition}
\begin{proof}
Let $w$ be the winning strategy for Player~2 that guarantees victory in at most $k+1$ moves.
Fix a set~$Z$, a $Z$-computable $\Psf$-instance~$X_0$ and $2^k \cdot (\ell-1)+1$ many $Z$-hyperimmune functions $(g_s)_{s < 2^k \cdot (\ell-1)+1}$. On the first move, Player~1 plays $X_0$. If $w(X_0)$ is an $X_0$-computable $\Psf$-solution~$S$ to~$X_0$, then all the $g$'s are $Z \oplus S$-hyperimmune, and we are done. So suppose $w(X_0)$ is an $X_0$-computable $\RT^2_2$-instance~$Y_1$, played by Player~2.
For $n > 1$, on the $n$\textsuperscript{th} move, suppose that $2^{k-n+2} \cdot (\ell-1)+1$ many~$g$'s are $(Z \oplus \bigoplus_{i < n-1} X_i)$-hyperimmune. Since $\RT^2_2$ preserves $2^{k-n+1} \cdot (\ell-1)+1$ among $2^{k-n+2} \cdot (\ell-1)+1$ hyperimmunities, Player~1 plays an $\RT^2_2$-solution~$X_{n-1}$ to~$Y_{n-1}$ such that $2^{k-n+1} \cdot (\ell-1)+1$ many $g$'s are $(Z \oplus \bigoplus_{i < n} X_i)$-hyperimmune. If $w(X_0,Y_1,\dots, X_{n-1})$ is an $(\bigoplus_{i < n} X_i)$-computable $\Psf$-solution~$S$ to~$X_0$, then since the winning strategy is in at most $k+1$ moves, $\ell$ many $g$'s are $(Z \oplus S)$-hyperimmune, and we are done. Otherwise $w(X_0,Y_1,\dots, X_{n-1})$ is a $(\bigoplus_{i < n} X_i$)-computable $\RT^2_2$-instance~$Y_n$ played by Player~2. 
Since~$w$ is a winning strategy for Player~2, then at some move $n \leq k+1$, the first case will hold for Player~2.
\end{proof}

We are now ready to prove \Cref{thm:brt3-not-omegak-rt22}.

\begin{proof}[Proof of \Cref{thm:brt3-not-omegak-rt22}]
Fix~$k \geq 0$. By \Cref{thm:off-diagonal-non-linear}, there is some~$\ell \geq 1$ such that there exists a triangle-free graph with $2^k \cdot (\ell-1)+1$ vertices having no independent set of size~$\ell$. Thus, by \Cref{cor:graph-brt-hyp-non-preserve}, $\BRT^3_{2,4}$ does not admit preservation of $\ell$ among $2^k \cdot (\ell-1)+1$ hyperimmunities. It follows by \Cref{prop:player-rt22-preservation} that $\BRT^3_{2,4} \not \leq_\omega^k \RT^2_2$.
\end{proof}

\subsection{$\Delta^0_2$ stable $\BRT^2_2$ preserves $\omega$ hyperimmunities}

Note that the instance of $\BRT^3_{2,4}$ witnessed in the proof of \Cref{thm:brt3-not-omegak-rt22} is obtained by constructing a $\Delta^0_2$ instance~$f$ of $\BRT^2_{2,3}$ such that any infinite set on which~$f$ is stable defeats sufficiently many hyperimmunities simultaneously. In this section, we prove that the stable counterpart of $\BRT^2_2$ for $\Delta^0_2$ instances can preserve countably many hyperimmunities simultaneously and has therefore no diagonalization power with respect to the framework of preservation hyperimmunities.

\begin{theorem}\label[theorem]{thm:delta2-stable-brt22-preserves-omega-hyp}
Let $f : [\NN]^2 \to 2$ be a $\Delta^0_2$ stable instance of $\BRT^2_2$ and $(h_n)_{n \in \NN}$ be a countable sequence of hyperimmune functions. There is an infinite $f$-homogeneous set~$G$ for color~$0$ such that for every~$n$, $h_n$ is $G$-hyperimmune.
\end{theorem}
\begin{proof}
Fix $f$ and $(h_n)_{n \in \NN}$. As in \Cref{thm:delta2-stable-brt22-weakly-low}, we may assume that for every~$x \in \NN$, $\lim_y f(x, y) = 0$. Let $k$ be such that there is no $f$-homogeneous set for color~1 of size~$k$.
\bigskip

\noindent
\textbf{Notion of forcing.}
Consider the notion of forcing whose conditions are Mathias conditions $(\sigma, X)$, where
\begin{itemize}
    \item[(1)] For every~$y \in X$, $\sigma \cup \{y\}$ is $f$-homogeneous for color~0;
    \item[(2)] $X$ is computably dominated.
\end{itemize}
Condition extension is the usual Mathias extension. Any filter~$\F$ induces a set~$G_\F = \bigcup_{(\sigma, X) \in \F} \sigma$.
\bigskip

\noindent
\textbf{Forcing infinity.}
We now prove that for every sufficiently generic filter~$\F$, $G_\F$ is infinite.

\begin{lemma}\label[lemma]{lem:delta2-stable-brt22-preserves-omega-hyp-infinity}
Let $p = (\sigma, X)$ be a condition and $y \in X$. There is some~$t \in \NN$ such that $(\sigma \cup \{y\}, X \cap (t, \infty))$ is a valid extension of~$p$. 
\end{lemma}
\begin{proof}
Pick~$y \in X$. Since $\lim_z f(y, z) = 0$, then there is some~$t \in \NN$ such that for every~$z > t$, $f(y, z) = 0$.
Since $X \cap (t, \infty)$ is $X$-computable, it is computably dominated, so $(\sigma \cup \{y\}, X \cap (t, \infty))$ is a valid extension of~$p$.
\end{proof}
\bigskip

\noindent
\textbf{Forcing question.}
Given $\sigma$ and $X$, let $\P(\sigma, X)$ be the $\Pi^0_1(X)$ class of all $\BRT^2_{2,k}$-instances~$g : [\NN]^2 \to 2$ such that for every~$y \in X$, $\sigma \cup \{y\}$ is $g$-homogeneous for color~0.

\begin{definition}
Given a condition $p = (\sigma, X)$ and a $\Sigma^0_1$-formula $\varphi(G)$, let $p \qvdash \varphi(G)$ hold if for every $g \in \P(\sigma, X)$, there is a finite set~$\rho \subseteq X$ which is $g$-homogeneous for color~0 and such that $\varphi(\sigma \cup \rho)$ holds.
\end{definition}

Note that by compactness, the forcing question is $\Sigma^0_1(X)$ uniformly in the parameters of the formula.

\begin{lemma}\label[lemma]{lem:delta2-stable-brt22-preserves-omega-hyp-forcing-question-spec}
Let $p$ be a condition and $\varphi(G)$ be a $\Sigma^0_1$-formula.
\begin{itemize}
    \item[(i)] If $p \qvdash \varphi(G)$, then there is an extension of~$p$ forcing $\varphi(G)$.
    \item[(ii)] If $p \nqvdash \varphi(G)$, then there is an extension of~$p$ forcing $\neg \varphi(G)$.
\end{itemize}
\end{lemma}
\begin{proof}
Say $p = (\sigma, X)$.
\begin{itemize}
    \item[(i)] Suppose $p \qvdash \varphi(G)$ holds. Then, letting~$g = f$, there is a finite set $\rho \subseteq X$ which is $f$-homogeneous for color~0 and such that $\varphi(\sigma \cup \rho)$ holds. Since~$f$ is stable, there is some~$t \in \NN$ such that for every~$x \in \rho$ and every~$y > t$, $f(x, y) = 0$. Then $(\sigma \cup \rho, X \cap (t, \infty))$ is an extension of~$p$ forcing~$\varphi(G)$.

    \item[(ii)] Suppose $p \nqvdash \varphi(G)$ holds. Let $\Q$ be the $\Pi^0_1(X)$ class of all~$g \in \P(\sigma, X)$ such that for every finite set $\rho \subseteq X$ which is $g$-homogeneous for color~0, $\varphi(\sigma \cup \rho)$ does not hold. By assumption, $\Q \neq \emptyset$, and by the computably dominated basis theorem (i.e., the hyperimmune-free basis theorem), there is some~$g \in \P(\sigma, X)$ such that $g \oplus X$ is computably dominated. Since~$\BRT^2_2$ is computably true, there is an infinite $g \oplus X$-computable $g$-homogeneous subset~$Y \subseteq X$.
    The condition $(\sigma, Y)$ is an extension of~$p$ forcing $\neg \varphi(G)$. 
\end{itemize}
\end{proof}

The next lemma states that the forcing question is $\Sigma^0_1$-compact.

\begin{lemma}\label[lemma]{lem:delta2-stable-brt22-preserves-omega-hyp-forcing-question-compact}
Let $p$ be a condition, and let $\varphi(G, x)$ be a $\Sigma^0_1$-formula.
If $p \qvdash \exists x \varphi(G, x)$, then there is some~$t \in \NN$ such that $p \qvdash (\exists x < t)\varphi(G, x)$.
\end{lemma}
\begin{proof}
Say $p = (\sigma, X)$. By compactness, $p \qvdash \exists x \varphi(G, x)$ holds if and only if there is a finite set~$F \subseteq X$ such that for every coloring $g : [F]^2 \to 2$ for which there is no $g$-homogeneous set for color~1 of size~$k$, there is a $g$-homogeneous set~$\rho \subseteq F$ such that $\exists x \varphi(\sigma \cup \rho, x)$ holds. Let $t$ be larger than all the (finitely many) witnesses $x$ of the existential statement. Then $p \qvdash (\exists x < t)\varphi(G, x)$.
\end{proof}
\bigskip

\noindent
\textbf{Preserving hyperimmunities.}
We now prove that the notion of forcing preserves any fixed hyperimmunity. Since the notion of forcing is not disjunctive, countably many hyperimmunities can be preserved simultaneously. Note that the lemma only exploits abstract properties of the forcing question, which we reprove for the sake of completeness.

\begin{lemma}\label[lemma]{lem:delta2-stable-brt22-preserves-omega-hyp-preservation}
Let $p$ be a condition and $e, n \in \NN$.
There is an extension of~$p$ forcing either that $\Phi_e^G$ is partial or that $\Phi_e^G$ does not dominate $h_n$.
\end{lemma}
\begin{proof}
Let $g : \NN \to \NN$ be the partial function which on input~$x$, searches for some~$t \in \NN$ such that $p \qvdash \Phi_e^G(x)\converge < t$ holds. If found, then $g(x) = t$, otherwise $g(x)$ is not defined. Since the forcing question is $\Sigma^0_1(X)$ uniformly in the parameters of the formula, the function $g$ is partial $X$-computable. We have two cases.
\begin{itemize}
    \item Case 1: $g$ is partial. Let $x$ be such that $g(x)$ is not defined. Unfolding the definition, for every~$t \in \NN$, $p \nqvdash \Phi_e^G(x)\converge < t$. Then by \Cref{lem:delta2-stable-brt22-preserves-omega-hyp-forcing-question-compact}, $p \nqvdash \Phi_e^G(x)\converge$, so by \Cref{lem:delta2-stable-brt22-preserves-omega-hyp-forcing-question-spec}(2), there is an extension of~$p$ forcing $\Phi_e^G(x)\diverge$.
    \item Case 2: $g$ is total. Since $h_n$ is hyperimmune and $X$ is computably dominated, $h_n$ is $X$-hyperimmune, so there is some~$x \in \NN$ such that $g(x) < h_n(x)$. Unfolding the definition, $p \qvdash \Phi_e^G(x)\converge < g(x)$. By \Cref{lem:delta2-stable-brt22-preserves-omega-hyp-forcing-question-spec}(1), there is an extension of~$p$ forcing $\Phi_e^G(x)\converge < g(x)$, hence forcing $\Phi_e^G(x)\converge < h_n(x)$.
\end{itemize}
\end{proof}
\bigskip

We are now ready to prove \Cref{thm:delta2-stable-brt22-preserves-omega-hyp}. Let $\F$ be a sufficiently generic filter for this notion of forcing. By definition of a condition, $G_\F$ is $f$-homogeneous for color~0.
By \Cref{lem:delta2-stable-brt22-preserves-omega-hyp-infinity}, $G_\F$ is infinite, and by \Cref{lem:delta2-stable-brt22-preserves-omega-hyp-preservation}, $h_n$ is $G_\F$-hyperimmune for every~$n \in \NN$. This completes the proof of \Cref{thm:delta2-stable-brt22-preserves-omega-hyp}.
\end{proof}

\section{$\BRT^3_{2,4}$ does not computably-reduce to $\RT^2$}

We prove:

\begin{reptheorem}{thm:brt3-not-omega-rt2}
$\BRT^3_{2,4} \not \leq_c \RT^2$.
\end{reptheorem}

To prove this, we construct a $\Delta^0_2$ instance $g : [\NN]^2 \to 2$ of $\BRT^2_{2,3}$ such that for any computable instance $f : [\NN]^2 \to k$ of $\RT^2$ (indeed, for any instance of $\RT^2$ in a suitably chosen Scott ideal), there is an $f$-homogeneous set $H$ such that $f \oplus H$ does not compute a $g$-homogenous set.  By \Cref{lem:delta2-stable-brt}, there is a stable computable instance $h : [\NN]^3 \to 2$ of $\BRT^3_{2,4}$ whose limit coloring is $g$.  Every $h$-homogeneous set is also $g$-homogeneous, and therefore there is an $f$-homogeneous set $H$ such that $f \oplus H$ does not compute an $h$-homogenous set.  Thus the $\BRT^3_{2,4}$ instance $h$ witnesses that $\BRT^3_{2,4} \not \leq_c \RT^2$.

The rest of this section is dedicated to constructing the desired $\Delta^0_2$ instance $g$ of $\BRT^2_{2,3}$.  It is convenient to think of $g$ as the characteristic function of the edge relation for an infinite triangle-free graph.

Let $\Phi_0, \Phi_1, \dots$ be a listing of all machines with the following properties for all $X$, $m$, and $n$.
\begin{itemize}
\item If $\Phi_e^X(m)\converge$, then $\Phi_e^X(m)$ is a set of size $m+2$ (and therefore must contain a number greater than $m$).
\item If $m < n$, $\Phi_e^X(m)\converge$, and $\Phi_e^X(n)\converge$, then $\Phi_e^X(m) \subseteq \Phi_e^X(n)$.
\end{itemize}
If an infinite set $Y$ is c.e.\ relative to some set $X$, then there is an $e$ such that $\Phi_e^X$ is total and $Y = \bigcup_m \Phi_e^X(m)$.

Let $\Sc = \{I_n : n \in \NN\}$ be a listing of a Scott ideal whose sets are uniformly computable in a fixed low set $C$.

Let $k \geq 2$.  We can view any set $X \subseteq \NN$ as the function $f_{X,k} : \NN \to k$ defined by
\begin{align*}
f_{X,k}(n) = 
\begin{cases}
0 & \text{if $X \cap [nk, nk+k-2] = \emptyset$}\\
i+1 & \text{if $nk+i$ is the least element of $X \cap [nk, nk+k-2]$}.
\end{cases}
\end{align*}
Of course $f_{X,k} \leq_T X$, and if $f : \NN \to k$, then there is an $X \leq_T f$ with $f_{X,k} = f$.  Via the pairing function, we can also view any $f_{X,k}$ and therefore $X$ as a function $[\NN]^2 \to k$.  When we write that an $f : [\NN]^2 \to k$ is in $\Sc$, we mean that $f = f_{X,k}$ for some $X \in \Sc$.

It is well-known that for every $k$, there is a finite triangle-free graph with chromatic number $>\! k$.  Let $(g_k : k \in \NN)$ be an effective listing of finite triangle-free graphs where $g_k$ is not $k$-colorable.  Let $|g_k|$ denote the size of $g_k$.

Consider a $k$-coloring $f : [\NN]^2 \to k$.  An $\emph{$f$-condition}$ is a pair $(\ol{E}, I)$ where $\ol{E} = E_0 \sqcup \cdots \sqcup E_{k-1}$ is a sequence of $k$ disjoint finite sets (and hence a partition of some finite set), $I$ is an infinite set in $\Sc$, and, for each $i < k$,
\begin{itemize}
\item $\max(E_i) < \min(I)$;
\item $f(x,y) = i$ for all $x \in E_i$ and $y \in I$.
\end{itemize}
Say that $f$-condition $(\ol{F}, J)$ extends $f$-condition $(\ol{E}, I)$ if $E_i \subseteq F_i \subseteq E_i \cup I$ for each $i < k$, and $J \subseteq I$.

We want a locally finite triangle-free graph $\Gamma \leq_T C' \equiv_T \emptyset'$ on  $\NN$ with no infinite independent set computable from $C$ as well as having the following property.
\begin{quote}
For every $k \geq 2$, every $k$-coloring $f : [\NN]^2 \to k$ in $\Sc$, every $f$-condition $(\ol{E}, I)$, and every $k$-tuple $\ol{e} = (e_0, \dots, e_{k-1})$, there is an extension of $(\ol{E}, I)$ that forces the requirement $\Q(\ol{e})$ described below.
\end{quote}

Requirements:
\begin{itemize}
\item[$\R(X, e)$:]  $\Phi_e^X$ is partial or there is an $m$ such that $\Phi_e^X(m)\converge$ is not an independent subset of $\Gamma$.

\smallskip

\item[$\Q(\ol{e})$:]  $\bigvee_{i < k} \R(C \oplus G_i, e_i)$.
\end{itemize}

We describe the basic steps used to define $\Gamma$, then we give the full construction of $\Gamma$, and then we show that every $f : [\NN]^2 \to k$ in $\Sc$ for every $k \geq 2$ has a homogeneous set $H$ where $C \oplus H$ does not compute an infinite independent subset of $\Gamma$.

Suppose we have defined a triangle-free $\Gamma$ on $[s]^2$ for some $s$.  Since we are working relative to $C'$, it is easy to extend $\Gamma$ to ensure that, given $e$, $\Phi_e^C$ does not yield an infinite independent subset of $\Gamma$.  Use $C'$ to ask if there is an $m$ such that $\Phi_e^C(m)$ contains distinct elements $x, y > s$.  If not, then $\Phi_e^C$ yields a finite set and we need not define $\Gamma$ further.  If so, let $\hat{s}$ be greater than $x$ and $y$.  Extend $\Gamma$ to $[\hat{s}]^2$ by making $(x,y)$ and edge.  No other pairs from $[\hat{s}]^2$ are edges in $\Gamma$ except those in $[s]^2$ that are already edges.  $\Gamma$ remains triangle-free because the new edge $(x,y)$ is between two fresh vertices.

Again suppose we have defined a triangle-free $\Gamma$ on $[s]^2$ for some $s$.  Let $f \in \Sc$ be a $k$-coloring, let $(\ol{E}, I)$ be an $f$-condition, and let $\ol{e} = (e_0, \dots, e_{k-1})$ be a $k$-tuple.  We show how to extend the definition of $\Gamma$ to ensure that $(\ol{E}, I)$ has an extension forcing $\Q(\ol{e})$.  To help do this, we define a finitely branching tree $T$ of auxiliary data of depth $|g_k|$.  The tree $T$ tells us how to extend $\Gamma$, and a certain path through $T$ tells us an extension of $(\ol{E}, I)$ forcing $\Q(\ol{e})$.  If that path has length less than $|g_k|$, then the extension forces $\Q(\ol{e})$ by divergence.  If that path has length $|g_k|$, then the extension forces $\Q(\ol{e})$ by convergence.

We describe the data recorded by $T$, then we explain how to define $T$ itself.  The following data is associated to each element $\sigma \in T$.
\begin{itemize}
\item A bound $b_\sigma$.
\item Disjoint subsets $\ol{P}(\sigma) = P_0(\sigma) \sqcup \cdots \sqcup P_{k-1}(\sigma)$ of $I \uh b_\sigma = \{z \in I : z < b_\sigma\}$.
\item Sets $W_i(\sigma) \subseteq P_i(\sigma)$ for each $i < k$.
\item A $J(\sigma) \subseteq I^{\geq b_\sigma} = \{z \in I : z \geq b_\sigma\}$ in $\Sc$ such that $f(x,y) = i$ whenever $i < k$, $x \in P_i(\sigma)$, and $y \in J(\sigma)$.  The set $J(\sigma)$ is not necessarily infinite, but it will be along the correct path.  Formally, we represent $J(\sigma)$ as an index $r$ such that $J(\sigma) = I_r$.
\end{itemize}
If $\sigma \in T$ is not the root $\epsilon$, then to $\sigma$ we additionally associate the following data.
\begin{itemize}
\item A label $\ell_\sigma < k$.
\item A number $m_\sigma$ and a witness $w_\sigma \in \Phi_{e_{\ell_\sigma}}^{E_{\ell_\sigma} \cup W_{\ell_\sigma}(\sigma)}(m_\sigma)\converge$ with $w_\sigma > m_\sigma$.  The bound $b_\sigma$ is greater than the use of this computation.
\end{itemize}

We maintain the following for $\sigma, \tau \in T$.
\begin{itemize}
\item If $|\sigma| < |\tau|$, then $b_\sigma < b_\tau$.  If $\sigma, \tau \neq \epsilon$, then also $m_\sigma < m_\tau$ and $w_\sigma < w_\tau$.
\item If $\sigma \prec \tau$, then $\ol{P}(\tau)$ extends $\ol{P}(\sigma)$ via $J(\sigma)$:  $P_i(\sigma) \subseteq P_i(\tau)$ and $P_i(\tau) \setminus P_i(\sigma) \subseteq J(\sigma)$ for all $i < k$.
\item If $\sigma \prec \tau$, then $W_i(\sigma) \subseteq W_i(\tau)$ and $W_i(\tau) \setminus W_i(\sigma) \subseteq P_i(\tau) \setminus P_i(\sigma)$ for all $i < k$.
\item If $\sigma \prec \tau$, then $J(\tau) \subseteq J(\sigma)$.
\end{itemize}

Start with the root $\epsilon \in T$.  The associated data is $b_\epsilon = 0$, $W_i(\epsilon) = P_i(\epsilon) = \emptyset$ for each $i < k$, and $J(\epsilon) = I$.  Notice that $J(\epsilon)$ vacuously has the required property because $P_i(\epsilon) = \emptyset$ for each $i < k$.

Define $T$ level-by-level.  Suppose that level $t$ has been defined.  Fix $m$ greater than $s$ and greater than $m_\sigma$ and $w_\sigma$ for all $\sigma \in T \setminus \{\epsilon\}$ already defined.  (In the $t=0$ case, just fix $m > s$.)  Do the following for each $\sigma \in T$ with $|\sigma| = t$ in lexicographic order.  Ask if for every partition $Q_0 \sqcup \cdots \sqcup Q_{k-1} = J(\sigma)$, there are $\ell < k$ and finite $V \subseteq Q_\ell$ such that $\Phi_{e_\ell}^{C \oplus (E_\ell \cup W_\ell(\sigma) \cup V)}(m)\converge$.  This is a $\Sigma^0_1$ property of $C$ by compactness, so $C'$ is able to answer.

If the answer is \qt{no,} then no children of $\sigma$ are added to $T$.  Thus $\sigma$ is a leaf in $T$.

If the answer is \qt{yes,} then there is a bound $b$ such that for every partition $\ol{Q} = Q_0 \sqcup \cdots \sqcup Q_{k-1}$ of $J(\sigma) \uh b$, there are $\ell < k$ and $V \subseteq Q_\ell$ such that $\Phi_{e_\ell}^{C \oplus (E_\ell \cup W_\ell(\sigma) \cup V)}(m)\converge$.  If necessary, increase $b$ so that it is greater the uses of all these computations and greater than $b_\tau$ for all $\tau \in T$ already defined.  List the partitions $(\ol{Q}(j) : j < n)$ of $J(\sigma) \uh b$ in lexicographic order.  For each $j < n$, add child $\sigma \cdot j$ to $T$ with the following data.
\begin{itemize}
\item $b_{\sigma \cdot j} = b$ and $m_{\sigma \cdot j} = m$.  Note that $b$ and $m$ were chosen to be respectively greater than all previously defined $b_\tau$ and $m_\tau$.  We build $T$ level-by-level, so $b_{\sigma \cdot j} > b_\tau$ and $m_{\sigma \cdot j} > m_\tau$ (if $\tau \neq \epsilon$) for all $\tau \in T$ with $|\tau| \leq t$.

\item $P_i(\sigma \cdot j) = P_i(\sigma) \cup Q_i(j)$ for each $i < k$.  $\ol{P}(\sigma)$ consists of disjoint subsets of $I \uh b_\sigma$ and $\ol{Q}(j)$ consists of disjoint subsets of $J(\sigma) \uh b_{\sigma \cdot j} \subseteq I^{\geq b_\sigma}\uh b_{\sigma \cdot j}$, so $\ol{P}(\sigma \cdot j)$ consists of disjoint subsets of $I \uh b_{\sigma \cdot j}$.

\item  Let $\ell < k$, $V \subseteq Q_\ell(j)$, and $w > m$ be such that $w \in \Phi_{e_\ell}^{C \oplus (E_\ell \cup W_\ell(\sigma) \cup V)}(m)\converge$.  Put $\ell_{\sigma \cdot j} = \ell$,  $W_\ell(\sigma \cdot j) = W_\ell(\sigma) \cup V$, and $w_{\sigma \cdot j} = w$.  Notice that $W_\ell(\sigma \cdot j) \setminus W_\ell(\sigma) = V \subseteq Q_i(j) = P_\ell(\sigma \cdot j) \setminus P_\ell(\sigma)$.  Also, $w > m$, and $m$ was chosen greater than all previously defined $w_\tau$.  Thus $w_{\sigma \cdot j}$ is greater than $w_\tau$ for all $\tau \in T \setminus \{\epsilon\}$ with $|\tau| \leq t$.

\item Put $W_i(\sigma \cdot j) = W_i(\sigma)$ for $i \neq \ell$.

\item Put $J(\sigma \cdot j) = \{y \in J(\sigma)^{\geq b} : \forall i < k, \forall x \in Q_i(j), (f(x,y) = i)\}$.  Suppose that $x \in P_i(\sigma \cdot j)$ for some $i < k$ and that $y \in J(\sigma \cdot j)$.  If $x \in P_i(\sigma)$, then $f(x,y) = i$ because $J(\sigma \cdot j) \subseteq J(\sigma)$.  If $x \in Q_i(j)$, then $f(x,y) = i$ by definition.  Note that $J(\sigma \cdot j) \in \Sc$ because $f$ and $J(\sigma)$ are in $\Sc$.  $C'$ can determine an index $r$ such that $J(\sigma \cdot j) = I_r$ because $J(\sigma \cdot j) = I_r$ is a $\Pi^0_1$ property of $C$.
\end{itemize}

Continue building $T$ until either defining level $|g_k|$ or reaching a level $t < |g_k|$ at which no extensions are made.  Recall that $\Gamma$ is currently a triangle-free graph on $[s]^2$.  Let $X = \{w_\sigma : \sigma \in T \setminus \{\epsilon\}\}$, and choose $\hat{s}$ greater than all the elements of $X$.  Note that the elements of $X$ are all greater than $s$.  Extend $\Gamma$ to $[\hat{s}]^2$ as follows.  Let $\{v_i : i < |g_k|\}$ denote the vertices of $g_k$.  Put an edge between $w_{\sigma}$ and $w_{\tau}$ in $\Gamma$ if and only if there is an edge between $v_{|\sigma|-1}$ and $v_{|\tau|-1}$ in $g_k$.  No other pairs from $[\hat{s}]^2$ are edges in $\Gamma$ except those in $[s]^2$ that are already edges.  $\Gamma$ remains triangle-free.  All new edges are between vertices $w_\sigma, w_\tau \in X$ with $|\sigma| \neq |\tau|$.  Thus a new triangle in $\Gamma$ would have to have vertices $w_\sigma, w_\tau, w_\eta$ with $0 < |\sigma| < |\tau| < |\eta|$.  However, this would correspond to a contradictory triangle $v_{|\sigma|-1}, v_{|\tau|-1}, v_{|\eta|-1}$ in $g_k$.

This extension of $\Gamma$ ensures that the $f$-condition $(\ol{E}, I)$ has an extension forcing $\Q(\ol{e})$.

We (non-effectively) trace a path through $T$.  Start at the root $\epsilon$, and note that $J(\epsilon) = I$ is infinite.  Suppose we have reached $\sigma \in T$ with $J(\sigma)$ infinite.  If $\sigma$ is not a leaf, let $b$ be the bound found when extending $\sigma$.  List $J(\sigma) \uh b$ as $x_0, \dots, x_{d-1}$.  Using $\mathsf{RT}^1$ at each step, inductively define $c : J(\sigma) \uh b \to k$ by letting $c(x_j)$ be the least $i$ such that there are infinitely many elements of $\{y \in J(\sigma) : \forall j' < j, (f(x_{j'}, y) = c(x_{j'}))\}$ with $f(x_j, y) = i$.  This yields a partition $\ol{Q} = Q_0 \sqcup \cdots \sqcup Q_{k-1}$ of $J(\sigma) \uh b$ where $x \in Q_i$ if and only if $c(x) = i$.  Move to the corresponding child $\sigma \cdot j$ of $\sigma$ where $\ol{Q}(\sigma \cdot j) = \ol{Q}$.  Note that $J(\sigma \cdot j)$ is infinite by the definition of $\ol{Q}$.  

Ultimately we reach a leaf $\sigma$ with $J(\sigma)$ infinite.  If $|\sigma| < |g_k|$, then $\sigma$ was not extended to the next level, so there is an $m$ and a partition $Q_0 \sqcup \cdots \sqcup Q_{k-1} = J(\sigma)$ in $\Sc$ such that for every $\ell < k$ and every finite $V \subseteq Q_\ell$, $\Phi_{e_\ell}^{C \oplus (E_\ell \cup W_{\ell}(\sigma) \cup V)}(m)\diverge$.  As $J(\sigma)$ is infinite, $Q_\ell$ is infinite for some $\ell < k$.  Let $\ol{E} \cup \ol{W}(\sigma)$ denote the sequence $(E_0 \cup W_0(\sigma)) \sqcup \cdots \sqcup (E_{k-1} \cup W_{k-1}(\sigma))$.  Let $i < k$.  If $x \in E_i$ and $y \in Q_\ell$, then $f(x,y) = i$ because $Q_\ell \subseteq J(\sigma) \subseteq I$ and $(E, I)$ is a condition.  If $x \in W_i(\sigma)$ and $y \in Q_\ell$, then $f(x,y) = i$ because $W_i(\sigma) \subseteq P_i(\sigma)$, $Q_\ell \subseteq J(\sigma)$, and we maintain that $f(x, y) = i$ whenever $x \in P_i(\sigma)$ and $y \in J(\sigma)$.  This means that $(\ol{E} \cup \ol{W}(\sigma), Q_\ell)$ is a condition.  Thus $(\ol{E} \cup \ol{W}(\sigma), Q_\ell)$ is a condition that extends $(\ol{E}, I)$ and forces $\R(C \oplus G_\ell, e_\ell)$ by divergence.

If $|\sigma| = |g_k|$, then consider the labels $\ell_{\sigma \uh i}$ and the witnesses $w_{\sigma \uh i}$ for $1 \leq i \leq |g_k|$.  Recall that $w_{\sigma \uh i}$ and $w_{\sigma \uh j}$ are adjacent in $\Gamma$ if and only if $v_i$ and $v_j$ are adjacent in $g_k$.  Therefore, viewing the labels as colors, there must be $1 \leq i < j \leq |g_k|$ where $w_{\sigma \uh i}$ and $w_{\sigma \uh j}$ are adjacent in $\Gamma$ and $\ell_{\sigma \uh i} = \ell_{\sigma \uh j}$ because $g_k$ is not $k$-colorable.  Let $\ell = \ell_{\sigma \uh i} = \ell_{\sigma \uh j}$.  Then $w_{\sigma \uh i}, w_{\sigma \uh j} \in \Phi_{e_\ell}^{C \oplus (E_\ell \cup W_\ell(\sigma))}(m_\sigma)$.  To see this, observe that
\begin{align*}
w_{\sigma \uh i} \in \Phi_{e_\ell}^{C \oplus (E_\ell \cup W_\ell(\sigma \uh i))}(m_{\sigma \uh i}) &\;=\; \Phi_{e_\ell}^{C \oplus (E_\ell \cup W_\ell(\sigma))}(m_{\sigma \uh i})\\
&\;\subseteq\; \Phi_{e_\ell}^{C \oplus (E_\ell \cup W_\ell(\sigma))}(m_\sigma),
\end{align*}
and likewise with $j$ in place of $i$.  The equality is because the elements of $W_\ell(\sigma) \setminus W_\ell(\sigma \uh i)$ are beyond the use of $\Phi_{e_\ell}^{C \oplus (E_\ell \cup W_\ell(\sigma \uh i))}(m_{\sigma \uh i})$ by the choice of bound $b_{\sigma \uh i}$.  The containment is by the convention on machines because $m_{\sigma \uh i} < m_\sigma$.  Thus $\Phi_{e_\ell}^{C \oplus (E_\ell \cup W_\ell(\sigma))}(m_\sigma)$ is not an independent subset of $\Gamma$.  Therefore the pair $(\ol{E} \cup \ol{W}(\sigma), J(\sigma))$ is a condition extending $(\ol{E}, I)$ as in the previous case, and it forces $\R(C \oplus G_\ell, e_\ell)$ by convergence.

To fully specify $\Gamma \leq_T C'$, list all tuples $((k_j, p_j, q_j, \ol{E}^j, \ol{e}^j) : j \in \NN)$, where $k_j \geq 2$, $p_j, q_j \in \NN$, $\ol{E}^j = E^j_0 \sqcup \cdots \sqcup E^j_{k_j-1}$ is a sequence of $k_j$ disjoint finite sets, and $\ol{e}^j = (e^j_0, \dots, e^j_{k_j-1})$ is a $k_j$-tuple.  Start at stage $0$ with $\Gamma$ the trivial graph on $[0]^2$.

At stage $2e+1$, $\Gamma$ has been defined up to $[s]^2$ for some $s$.  Extend $\Gamma$ up to $[\hat{s}]$ for some $\hat{s} > s$ to ensure that $\Phi_e^C$ is not an infinite independent subset of $\Gamma$.

At stage $2j+2$, $\Gamma$ has been defined up to $[s]^2$ for some $s$.  Consider the tuple $(k_j, p_j, q_j, \ol{E}^j, \ol{e}^j)$, and let $k = k_j$, $f = f_{I_{p_j}, k}$, $I = I_{q_j}$, $\ol{E} = \ol{E}^j$, and $\ol{e} = \ol{e}^j$.  View $(\ol{E}, I)$ as a potential $f$-condition.  Check that $I$ has no element below $\max \bigcup_{i<k}E_i$, and ask $C'$ if $f(x,y) = i$ whenever $i < k$, $x \in E_i$, and $y \in I$.  If the check fails, then $(\ol{E}, I)$ is not an $f$-condition, and we proceed to the next stage.  If the check passes, then $(\ol{E}, I)$ looks like an $f$-condition, except $I$ may be finite.  Extend $\Gamma$ up to $[\hat{s}]^2$ for some $\hat{s} > s$ via the procedure above for $f$, $(\ol{E}, I)$, and $\ol{e}$.  The procedure produces an extension of $\Gamma$ regardless of whether $I$ is infinite.  If $I$ is indeed infinite, then the extended definition of $\Gamma$ ensures that there is an extension of the $f$-condition $(\ol{E}, I)$ forcing requirement $\Q(\ol{e})$.

Our $\Gamma$ is triangle-free because it is triangle-free at every step.  Also, once $\Gamma$ is defined up to $[s]^2$, then no further edges are ever incident to vertices $x < s$.  Thus $\Gamma$ is locally finite.

Now let $f = f_{I_n, k}$ be a $k$-coloring in $\Sc$ for some $k \geq 2$.  We show that $f$ has a homogeneous set $H$ such that $C \oplus H$ computes no infinite independent subset of $\Gamma$.  Let $(\ol{E}, I)$ be an $f$-condition minimizing the size of $X = \{c < k : \exists x \in I, \exists^\infty y \in I, (y > x \land f(x,y) = c)\}$.  If $J \in \Sc$ is an infinite subset of $I$ and $c \in X$, then there is an $x \in J$ where $f(x,y) = c$ for infinitely many $y \in J$.  Otherwise $(\ol{E}, J)$ is a condition contradicting the minimality of $X$.

If $X$ is a singleton, say $X = \{c\}$, then $I$ is limit homogeneous for $f$ with color $c$.  Thus $f \oplus I \leq_T C$ computes an $f$-homogenous set $H$.  The graph $\Gamma$ does not have an infinite independent set computable from $C$, so $C \oplus H \equiv_T C$ cannot compute an infinite independent subset of $\Gamma$.

Otherwise $X = \{c_0, \dots, c_{\hat{k}-1} \}$ for some $\hat{k} \geq 2$.  Let $\hat{f} : [\NN]^2 \to \hat{k}$ be given by
\begin{align*}
\hat{f}(x,y) =
\begin{cases}
0 & \text{if $f(x,y) = c_0$ or $f(x,y) \notin X$}\\
i & \text{if $f(x,y) = c_i$ for $i > 0$.}
\end{cases}
\end{align*}
Note that $\hat{f} \leq_T f$ and therefore that $\hat{f} \in \Sc$.  Let $F_i = E_{c_i}$ for each $i < \hat{k}$.  Then $(\ol{F}, I)$ is an $\hat{f}$-condition.  If an infinite $G$ with $F_i \subseteq G \subseteq F_i \cup I$ is $\hat{f}$-limit homogeneous with color $i$, then $G$ is also $f$-limit homogeneous with color $c_i$.  This is clear for $0 < i < \hat{k}$.  It also holds for $i = 0$.  If $x \in F_0$, then $f(x,y) = c_0$ for all $y \in I$.  If $x \in G \cap I$ and $\hat{f}(x,y) = 0$ for almost every $y \in G$, then it must also be that $f(x,y) = c_0$ for almost every $y \in G$ because there are only finitely many $y \in I$ with $f(x,y) \notin X$.  Furthermore, if $J \in \Sc$ is an infinite subset of $I$ and $i < \hat{k}$, then there is an $x \in J$ where $\hat{f}(x,y) = i$ for infinitely many $y \in J$.  Therefore we may define a sequence of extensions
\begin{align*}
(\ol{F}, I) = (\ol{F}^0, I^0) \geq (\ol{F}^1, I^1) \geq (\ol{F}^2, I^2) \geq \cdots
\end{align*}
where for every $\hat{k}$-tuple $\ol{e}$ there is an $\hat{f}$-condition $(\ol{F}^n, I^n)$ forcing $\Q(\ol{e})$, and for every $i < \hat{k}$ and $s$, there is an $\hat{f}$-condition $(\ol{F}^n, I^n)$ with $|F^n_i| > s$.  Let $G_i = \bigcup F^n_i$ for each $i < \hat{k}$.  Each $G_i$ is infinite, $\hat{f}$-limit homogeneous for color $i$, and thus $f$-limit homogeneous for color $c_i$.  By Lachlan's disjunction, meeting requirement $\Q(\ol{e})$ for every $\hat{k}$-tuple $\ol{e}$ ensures that there is an $i < \hat{k}$ such that requirement $\R(C \oplus G_i, e)$ is met for every $e$.  Thus $C \oplus G_i$ does not compute an infinite independent subset of $\Gamma$.  As $f \leq_T C$ and $G_i$ is $f$-limit homogeneous, there is an $f$-homogeneous set $H \leq_T C \oplus G_i$.  Thus there is an $f$-homogeneous set $H$ where $C \oplus H$ does not compute an infinite independent subset of $\Gamma$.  This completes the proof of \Cref{thm:brt3-not-omega-rt2}.

\section*{Acknowledgements}
The authors thank Emanuele Frittaion for helpful discussions.  This project was partially supported by EPSRC grant EP/T031476/1.

\bibliography{BRT3}
\bibliographystyle{plain}

\end{document}